\documentclass[11pt,reqno]{amsart}
\usepackage[all]{xy}
\usepackage{amssymb}
\usepackage{amsthm}
\usepackage[normalem]{ulem}
\usepackage{amsmath,mathtools}
\usepackage{amscd,enumitem}
\usepackage{verbatim}
\usepackage{eurosym}
\usepackage{float}
\usepackage{color}
\usepackage{dcolumn}
\usepackage[mathscr]{eucal}
\usepackage[all]{xy}
\usepackage{bbm}
\usepackage[textheight=8.5in, textwidth=6.7in]{geometry}
\newtheorem*{conj*}{Conjecture}
\newtheorem{theorem}{Theorem}[section]

\theoremstyle{definition}
\newtheorem*{remark}{Remark}
\theoremstyle{plain}

\newtheorem{lemma}[theorem]{Lemma}

\newtheorem{prop}[theorem]{Proposition}

\newtheorem{corollary}[theorem]{Corollary}
\newcommand{\ord}{\mathrm{ord}}
\newcommand{\legendre}[2]{\genfrac{(}{)}{}{}{#1}{#2}}
\newcommand{\Z}{\mathbb{Z}}
\newcommand{\Q}{\mathbb{Q}}
\newcommand{\R}{\mathbb{R}}
\newcommand{\N}{\mathbb{N}}

\newcommand{\SL}{\operatorname{SL}}
\newcommand{\C}{\mathbb{C}}

\newcommand{\re}[1]{\text{Re}\(#1\)}

\renewcommand{\pmod}[1]{\,\,({\rm mod}\,\,{#1})}
\DeclareMathOperator\Log{Log}
\numberwithin{equation}{section}

\newtheoremstyle{example}
  {\topsep}   
  {\topsep}   
  {\normalfont}  
  {0pt}       
  {\bfseries} 
  {.}         
  {5pt plus 1pt minus 1pt} 
  {}          
\theoremstyle{example}
\newtheorem*{example}{Example}

\def\({\left(}
\def\){\right)}

\def\lp{\left(}
\def\rp{\right)}

\usepackage{centernot}

\begin{document}
\title[Partition distributions]{Distributions on partitions arising from Hilbert schemes and hook lengths}
\author{Kathrin Bringmann}
\address{Department of Mathematics and Computer Science\\Division of Mathematics\\University of Cologne\\ Weyertal 86-90 \\ 50931 Cologne \\Germany}
\email{kbringma@math.uni-koeln.de}

\author{William Craig}

\author{Joshua Males}
\address{Department of Mathematics, Machray Hall, University of Manitoba, Winnipeg,
	Canada}
\email{joshua.males@umanitoba.ca}

\author{Ken Ono}
\address{Department of Mathematics, University of Virginia, Charlottesville, VA 22904}
\email{wlc3vf@virginia.edu}
\email{ko5wk@virginia.edu}


\begin{abstract} 
Recent works at the interface of algebraic combinatorics, algebraic geometry, number theory, and topology
have provided new integer-valued invariants on integer partitions.   It is natural to consider the distribution of partitions when sorted by these invariants in congruence classes.
We consider the prominent situations which arise from
extensions of the Nekrasov--Okounkov hook product formula, and from Betti numbers of various Hilbert schemes of $n$ points on $\C^2.$  For the Hilbert schemes, we prove that homology is equidistributed as $n\to \infty.$ For  $t$-hooks, we prove distributions which are often not equidistributed. The cases where $t\in \{2, 3\}$ stand out, as there are congruence classes
where such counts are zero. To obtain these distributions, we obtain analytic results which are of independent interest. We
determine the asymptotics, near roots of unity, of the ubiquitous infinite products
$$
F_1(\xi; q):=\prod_{n=1}^{\infty}\left(1-\xi q^n\right), \ \ \
F_2(\xi; q):=\prod_{n=1}^{\infty}\left(1-(\xi q)^n\right) \ \ \  {\text {\rm and}}\ \ \ 
F_3(\xi; q):=\prod_{n=1}^{\infty}\left(1-\xi^{-1}(\xi q)^n\right).
$$
\end{abstract}

\maketitle
\section{Introduction and statement of results}

A {\it partition} of a non-negative integer $n$, denoted $\lambda \vdash n,$ is any nonincreasing sequence of positive integers, say 
 $\lambda=(\lambda_1, \lambda_2, \dots, \lambda_m)$, that satisfies $|\lambda|:=\lambda_1+\dots+\lambda_m= n$. As usual, we let $p(n)$ denote the number of such partitions. One hundred years ago,
Hardy and Ramanujan \cite{HardyRamanujan} proved their striking asymptotic formula 
\begin{equation}\label{HR}
p(n)\sim \frac{1}{4\sqrt{3}n} \cdot e^{\pi \sqrt{\frac{2n}{3}}},
\end{equation}
as $n\rightarrow \infty.$ This work marked the birth of the so-called Circle Method. 

Partitions appear in many areas of mathematics. We consider recently discovered structures which arise at the interface of algebraic combinatorics, algebraic geometry, number theory,  and topology, where the size $n$ partitions play a prominent role in defining various integer-valued invariants. These invariants can be sorted by congruence conditions, resulting in identities
of the form
\begin{equation}\label{decomposition}
p(n)= C(0,b;n) + C(1,b;n)+\dots+C(b-1,b;n),
\end{equation}
where $C(a,b;n)$ counts those partitions whose invariant is in the congruence class $a\pmod b.$
In the spirit of Dirichlet's Theorem on primes, where primes are equidistributed over admissible congruence classes, one may ask how the partitions are distributed, as $n\to \infty,$ over the arithmetic progressions modulo $b.$ We answer these questions for $t$-hooks, which arise in work of Han \cite{Han} that refined the Nekrasov--Okounkov hook product formula, and for Betti numbers
of various Hilbert schemes on $n$ points in $\C^2,$ as established by G\"ottsche \cite{Gottsche, GottscheICM}, and Buryak, Feigin, and Nakajima \cite{IMRN, Gottsche, GottscheICM}.

We first consider the distribution of $t$-hooks.
Each partition has a {\it Ferrers--Young diagram}
$$
\begin{matrix}
\bullet & \bullet & \bullet & \dots & \bullet & \leftarrow & \lambda_1 \text{ many nodes}\\
\bullet & \bullet & \dots & \bullet & & \leftarrow & \lambda_2 \text{ many nodes}\\
\vdots & \vdots & \vdots & &  &  \\
\bullet & \dots & \bullet & & & \leftarrow & \lambda_m \text{ many nodes},
\end{matrix}
$$
and each node has a {\it hook length}. The node in row $k$ and column $j$ has hook length
$h(k,j):=(\lambda_k-k)+(\lambda'_j-j)+1,$ where $\lambda'_j$ is the number of nodes in column $j$.
These numbers play many significant roles in combinatorics, number theory, and representation theory
(for example, see \cite{KerberJames,rains}).

We investigate those hook lengths which are multiples of a fixed positive integer $t$, the so-called {\it $t$-hooks}.
We let $\mathcal{H}_t(\lambda)$ denote the multiset of $t$-hooks of a partition $\lambda$.
In recent work,  the second author and Pun \cite{CraigPun} analyzed the $t$-hook partition functions
\begin{align*}
p_t^e(n):=\# \{ \lambda \vdash n   \ : \   \# \mathcal{H}_t(\lambda) \ {\text {\rm is even}} \}, \quad
p_t^o(n):=\# \{ \lambda \vdash n  \ : \   \# \mathcal{H}_t(\lambda) \ {\text {\rm is odd}}\},
\end{align*}
which divide the partitions of $n$ into two subsets, those with an even (resp. odd) number of $t$-hooks.
For even $t$, they proved that partitions are equidistributed  between these subsets as $n\rightarrow \infty.$ 
Namely, they showed that
$$
\lim_{n\rightarrow \infty} \frac{p_t^e(n)}{p(n)}=\lim_{n\rightarrow \infty}\frac{p_t^o(n)}{p(n)}=\frac{1}{2}.
$$
However, for odd $t$ they found that the partitions are not equidistributed.
More precisely, if $a\in \{0, 1\},$ then they proved that\footnote{This claim is trivial if $t=1$
as $p_1^e(n)=p(n)$ (resp. $p_1^o(n)=p(n)$) if $n$ is even (resp. odd).}
\begin{equation*}\label{wierd}
\lim_{n\rightarrow \infty} \frac{p_t^e(2n+a)}{p(2n+a)}=
\frac{1}{2}+\frac{(-1)^a}{2^{\frac{1}{2}(t+1)}}.
\end{equation*}

In view of this unexpected result, it is natural to consider the
 more general $t$-hook partition functions
\[
p_t(a,b;n):=\# \{ \lambda \vdash n \ : \ \# \mathcal{H}_t(\lambda)\equiv a \pmod b\}.
\]
The $p_t(a,b;n)$ are clear generalizations of $p_t^e(n)$ and $p_t^o(n)$. In this setting, (\ref{decomposition}) is
$$
p(n)=p_t(0,b;n) + p_t(1,b;n) + p_t(2,b;n) +\dots + p_t(b-1,b;n).
$$
For odd primes $b$, we determine the distribution of these decompositions as $n\rightarrow \infty$, and
in many situations they turn out to be non-uniform.
To this end, we first obtain asymptotic formulas for $p_t(a,b;n)$. For this, we define a modified indicator function $\mathbb{I}$ by
\begin{equation}\label{eqn: defn of I}
\mathbb{I}(a,b,t,n):= \begin{cases} b-1 & \text{if } \frac{1}{24} \left(1-t^2\right)\left(1-b^2\right) + at - n \equiv 0 \pmod{b}, \\ -1 & \text{otherwise}, \end{cases}\\
\end{equation}
and a distribution function
\begin{align} \label{XXX}
c_t(a,b;n)&:= \frac 1b + \begin{cases} 0 & \text{ if } b | t, \\[+.1in] (-1)^{\frac{(1-t)(b-1)}{4}} \mathbb{I}(a,b,t,n) b^{-\frac{t+1}{2}} \left( \frac{t}{b} \right) & \text{ if } b \centernot | t \text{ and } t \text{ is odd,} \\[+.1in] i^{\frac{(1-t)(b-1)}{2}} \varepsilon_b b^{-\frac t2} \left(\frac{\frac{1}{24}\left(1-t^2\right)\left(1-b^2\right)+at-n}{b}\right)  & \text{ if } b \centernot | t \text{ and } t \text{ is even}, \end{cases}
\end{align}
 where $(\frac{\bullet}{b})$ is the Legendre symbol, and $\varepsilon_d:=1$ if $d\equiv 1 \pmod{4}$ and $\varepsilon_d:=i$ if $d\equiv 3 \pmod{4}$. This function exactly characterizes the distribution properties of the $p_t(a,b;n)$. In particular, the second summand in (\ref{XXX}) represents the obstruction to equidistribution.

We prove the following asymptotic formulae for $p_t(a,b;n).$ 

\begin{theorem}\label{Theorem2}
If $t>1$, $b$ is an odd prime, and $0\leq a<b ,$  then as $n\rightarrow \infty$ we have
$$
p_t(a,b;n)\sim \frac{c_t(a,b;n)}{4\sqrt{3}n}\cdot  e^{\pi \sqrt{\frac{2n}{3}}}.
$$
\end{theorem}

\begin{remark}
Thanks to equation  \eqref{Eqn: exact formula} in the proof of Theorem \ref{Theorem2}, we  actually obtain an exact formula for $p_t(a,b;n)$ as a complicated convergent infinite sum. 
\end{remark}

As a corollary, we obtain the following limiting distributions.

\begin{corollary}\label{Corollary3}
Assuming the hypotheses in Theorem \ref{Theorem2}, if $0\leq a_1<b$ and $0\leq a_2 <b,$  then 
$$
		\lim_{n\rightarrow \infty}\frac{p_t(a_1, b; b n+a_2)}{p(b n+a_2)}=c_t(a_1,b;a_2).
$$
\end{corollary}

\begin{example}
For $4$-hooks with $b=3$, the collection of values $c_4(a_1,3;a_2)$ in Corollary~\ref{Corollary3} implies that
$$
\lim_{n\rightarrow \infty} 
\frac{p_4(a,3;3n)}{p(3n)}=\begin{cases} \frac49 \ \ \ \ \ &{\text {\rm if $a=0$,}}\\
\frac13 \ \ \ \ \ &{\text {\rm if $a=1,$}}\\
\frac29 \ \ \ \ \ &{\text {\rm if $a=2$.}}
\end{cases}
$$
Further examples are offered in Section~\ref{Section5}.
\end{example}

The cases where $t\in \{2, 3\}$ are particularly striking. In addition to many instances of non-uniform distribution, there are situations where certain counts are actually identically zero.

\begin{theorem}\label{Vanishing}
The following are true.

\begin{enumerate}[leftmargin=*]
\item[\normalfont(1)] If $\ell$ is an odd prime and $0\leq a_1, a_2<\ell$ satisfy
$(\frac{-16a_1+8a_2+1}{\ell})=-1,$ then for every non-negative integer $n$ we have
$$
p_2(a_1,\ell;\ell n+a_2)=0.
$$

\item[{\normalfont(2)}] If $\ell\equiv 2\pmod 3$ is prime and $0\leq a_1, a_2<\ell^2$ have the property that
$\ord_{\ell}(-9a_1+3a_2+1)=1$, then
for every non-negative integer $n$ we have
$$
p_3\left(a_1,\ell^2;\ell^2 n+a_2\right)=0.
$$
\end{enumerate}
\end{theorem}

\begin{example}For $\ell=3,$ Theorem~\ref{Vanishing} (1) implies that
$$
p_2(0,3; 3n+2)=p_2(1, 3; 3n+1)=p_2(2,3;3n)=0.
$$
More generally, for every odd prime $\ell$ and each $0\leq a_1<\ell$, there are
$\frac{1}{2}(\ell-1)$ choices of $0\leq a_2<\ell$ satisfying the given hypotheses. In particular, there
are $\frac{1}{2}(\ell^2-\ell)$ many pairs of $a_1$ and $a_2$ giving rise to vanishing arithmetic progressions for $2$-hooks.
\end{example}

\begin{example}
For $\ell=2$, Theorem~\ref{Vanishing} (2) gives
$$
p_3(0,4;4n+3)=p_3(1,4;4n+2)=p_3(2,4;4n+1)=p_3(3,4;4n)=0.
$$
Moreover, for each $\ell$ and each $0\leq a_1<\ell^2$, there are $\ell-1$ choices for $a_2.$
\end{example}

\begin{remark} Theorem~\ref{Vanishing} depends on the paucity of $2$-core and $3$-core partitions. Recall that a partition $\lambda$ is a $t$-{\it core} if $\mathcal{H}_{t}(\lambda)=\emptyset.$
There are no such vanishing results for $t\geq 4$. This follows from the proof of the $t$-core conjecture
by Granville and the fourth author \cite{GranvilleOno}.
McSpirit and Scheckelhoff \cite{MS} have found a beautiful combinatorial proof of Theorem~\ref{Vanishing} which makes use of the theory of abaci, $t$-cores and $t$-quotients.
\end{remark}

We now turn to applications of partitions in  algebraic geometry and topology. 
The fundamental goal of topology  is to determine whether two spaces have the same topological, differentiable, or complex analytic structure. 
One seeks
invariants that distinguish dissimilar spaces.
For complex manifolds, the Hodge numbers  are one class of invariants. For any $n$-dimensional complex manifold $M$ and any $0 \leq s,t,\leq n$, the Hodge number $h^{s,t}(M)$ gives the dimension of a certain vector space of differential forms on $M$. 
For the manifolds we consider, the Betti numbers arise as linear combinations of the Hodge numbers
(for example, see \cite{Wells}). We shall determine the asymptotics and modular distribution properties
of certain Betti numbers. 

We consider examples occurring in the algebraic geometry of
Hilbert schemes (for example, see \cite{Hilbert_schemes}).
 The $n$-th  \textit{Hilbert scheme} of a projective variety $S$ is a projective variety $\mathrm{Hilb}^n(S)$ that is a ``smoothed'' version of the $n$-th symmetric product of $S$ (for example, see \cite{GottscheICM, Hilbert_schemes}).
The $n$-th symmetric product of a manifold $M$ admits a simple combinatorial interpretation: outside of a negligible subset, the symmetric product is the collection of subsets of $M$ of size $n$ assembled as a manifold on its own.
Rather nicely, the Hodge numbers of a complex projective surface $S$ determine the Hodge numbers of $\mathrm{Hilb}^n(S)$
in a beautiful combinatorial way.
 This is captured by the pleasing formula of G\"ottsche \cite{Gottsche, GottscheICM}
\begin{equation*} \sum_{n,s,t}(-1)^{s+t}h^{s,t}(\mathrm{Hilb}^n(S))x^{s-n}y^{t-n}q^n =\prod_{n =1}^\infty \frac{\prod_{s +t \mathrm{\ odd}} \left(1- x^{s-1}y^{t-1}q^n\right)^{h^{s,t}(S)}}{\prod_{s +t \mathrm{\ even}} \left(1- x^{s-1}y^{t-1}q^n\right)^{h^{s,t}(S)}}. \label{eq:goettsche} 
\end{equation*}
These $q$-infinite products often essentially specialize to modular forms, which then leads to asymptotics and distribution results via a standard application of the Circle Method. Indeed, the fourth author and his collaborators carried this out in \cite{GGOR}. 
Here we consider a prominent situation involving partitions, where modular forms do not arise, a fact which complicates the computation of asymptotics and distributions.
Namely, we investigate the
 Hilbert schemes that arise from $n$ points on $\C^2$ that
have been considered recently by G\"ottsche \cite{Gottsche, GottscheICM}, and Buryak, Feigin, and Nakajima \cite{BuryakFeigin, IMRN}. 

We denote the Hilbert scheme of $n$ points of $\C^2$
by $(\C^2)^{[n]}.$ For $0\leq a<b$, we consider the modular sums of Betti numbers
\begin{equation*}
B\left(a,b; \left(\C^2\right)^{[n]}\right):=\sum_{j\equiv a\pmod b} b_j\left(\left(\C^2\right)^{[n]}\right)=
\sum_{j\equiv a\pmod b} \dim \left(H_j\left(\left(\C^2\right)^{[n]},\Q\right)\right).
\end{equation*}
We also consider their quasihomogeneous versions.  To define them, we use the torus  $(\C^{\times})^2$-action on $\C^2$ defined by scalar multiplication (i.e., $(t_1, t_2)\cdot (x,y):=(t_1x, t_2 y)$). This action lifts to $(\C^2)^{[n]}.$ For relatively prime $\alpha, \beta\in \N,$ we 
let $T_{\alpha,\beta}:=\{(t^{\alpha}, t^{\beta}) \ : \ t\in \C^{\times}\},$ a one-dimensional subtorus.
The quasihomogeneous Hilbert scheme $((\C^2)^{[n]})^{T_{\alpha,\beta}}$ is the fixed point set
of  $(\C^2)^{[n]}.$ We consider their modular sums of Betti numbers
\begin{align*}
B\left(a,b; \left(\left(\C^2\right)^{[n]}\right)^{T_{\alpha,\beta}}\right):=\sum_{j\equiv a\pmod b} b_j\left(\left(\left(\C^2\right)^{[n]}\right)^{T_{\alpha,\beta}}\right)=
\sum_{j\equiv a\pmod b} \dim \left(H_j\left(\left(\left(\C^2\right)^{[n]}\right)^{T_{\alpha,\beta}},\Q\right)\right).
\end{align*}
\begin{remark}
The odd index Betti numbers for these Hilbert schemes are always zero. In fact, for $a$ odd and $b$ even, simple calculations using Corollary \ref{HilbertModularGeneratingFunctions} reveal that both $B(a,b; (\C^2)^{[n]})$ and\\$B(a,b;((\C^2)^{[n]})^{T_{\alpha,\beta}})$ identically vanish.
Moreover, in accord with (\ref{decomposition}), we have the homology decompositions for $p(n)$
\begin{equation}\label{decompositionBetti}
p(n)=\sum_{a=0}^{b-1} B\left(a,b; \left(\C^2\right)^{[n]}\right)
=\sum_{a=0}^{b-1} B\left(a,b; \left(\left(\C^2\right)^{[n]}\right)^{T_{\alpha,\beta}}\right).
\end{equation}
\end{remark}
These results require the rational numbers
\begin{equation} \label{XXXX}
	d(a,b):=\begin{cases} \frac{1}{b} \ \ \ \ \ &{\text {\rm if $b$ is odd,}}\\
	                       \frac{2}{b} \ \ \ \ \ &{\text {\rm if $a$ and $b$ are even,}}\\
	                       0 \ \ \ \ \ &{\text {\rm if $a$ is odd and $b$ is even.}}
	                       \end{cases}
\end{equation}

\begin{theorem}\label{Theorem4}
Assuming the notation above, the following are true.
\begin{enumerate}[leftmargin=*]
	\item[\rm (1)] As $n\rightarrow \infty$, we have
	$$
	B\left(a,b; \left(\C^2\right)^{[n]}\right)\sim \frac{d(a,b)}{4\sqrt{3}n}\cdot  e^{\pi \sqrt{\frac{2n}{3}}}.
	$$
	
	\item[\rm (2)] If $\alpha, \beta\in \N$ are relatively prime, then as $n\rightarrow \infty$ we have 
	$$
	B\left(a,b; \left(\left(\C^2\right)^{[n]}\right)^{T_{\alpha,\beta}}\right)\sim \frac{d(a,b)}{4\sqrt{3}n} \cdot e^{\pi \sqrt{\frac{2n}{3}}}.
	$$
\end{enumerate}
\end{theorem}

As a consequence of Theorem~\ref{Theorem4}, we obtain distributions (i.e., see (\ref{decompositionBetti})) for the proportions
\begin{equation*}
\delta(a,b;n):=
\frac{B\left(a,b; \left(\C^2\right)^{[n]}\right)}{p(n)}\  \ \ \ {\text {\rm and}}\ \ \ \
\delta_{\alpha,\beta}(a,b;n):=\frac{B\left(a,b; \left(\left(\C^2\right)^{[n]}\right)^{T_{\alpha,\beta}}\right)}{p(n)}.
\end{equation*}

\begin{corollary}\label{Corollary5} If $0\leq a<b$, then
the following are true. 
\begin{enumerate}[leftmargin=*]
	\item[\rm (1)] We have that
	$$
	\lim_{n\rightarrow \infty} \delta(a,b;n)=d(a,b).
	$$
	
	\item[\rm (2)] If $\alpha, \beta\in \N$ are relatively prime, then we have
	$$
	\lim_{n\rightarrow \infty} 
	\delta_{\alpha,\beta}(a,b;n)=d(a,b).
	$$
\end{enumerate}
\end{corollary}

This paper is organized as follows. In Section~\ref{Section2}, we state and prove a general theorem
(see Theorem~\ref{Theorem1})  on the asymptotic properties (near roots of unity) of the three infinite products given in the abstract, a result that is of independent interest. The proof is obtained by suitably adapting the method of Euler--Maclaurin summation in two cases, and via modularity in the other. In Section~\ref{Section3} we recall recent work of Han
extending the Nekrasov--Okounkov partition formula, and we prove Theorems~\ref{Theorem2} and \ref{Vanishing}.  To show Theorem~\ref{Theorem2} and Corollary~\ref{Corollary3}, we employ Theorem~\ref{Theorem1} (2) and results of Zuckerman pertaining to exact formulas for Fourier coefficients of modular forms. In Section~\ref{Section4} we recall the work of G\"ottsche,  and Buryak, Feigin, and Nakajima on homogeneous and quasihomogeneous Hilbert schemes for $n$ points, which we then employ to prove Theorem~\ref{Theorem4} and Corollary~\ref{Corollary5} using Theorem~\ref{Theorem1} (1), and (3), and results of Ngo--Rhoades using Wright's Circle Method.  Finally, in Section~\ref{Section5} we offer numerical examples of these results.

\section*{Acknowledgements}

The authors thank George Andrews, Walter Bridges, Giulia Cesana, Johann Franke, Jack Morava, and Ole Warnaar for helpful discussions related to the results in this paper. Moreover we thank the referees for helpful comments. The first author has received funding~from the European Research Council (ERC) under the European Union’s Horizon 2020 research and innovation programme (grant agreement No. 101001179). The research of the third author conducted for this paper is supported by the Pacific Institute for the Mathematical Sciences (PIMS). The research and findings may not reflect those of the Institute. The fourth author thanks the support of the Thomas Jefferson Fund and the NSF (DMS-1601306 and DMS-2055118), and the Kavli Institute grant NSF PHY-1748958.

\section{Asymptotics for special $q$-infinite products}\label{Section2}

The Hardy--Ramanujan asymptotic formula given in \eqref{HR} marked the birth of the Circle Method. Its proof relied  on the modular transformation properties of {\it Dedekind's eta-function} $\eta(\tau):= q^{\frac{1}{24}} \prod_{n=1}^{\infty}(1-q^n),$ where
$q:=e^{2\pi i \tau}$ (for example, see Chapter 1 of \cite{CBMS}). Their work has been thoroughly developed in the theory of modular forms and harmonic Maass forms (for example, see Chapter~15 of \cite{BFOR}), and
has been generalized beyond this setting in papers by  Grosswald, Meinardus, Richmond, Roth, and Szekeres \cite{Grosswald, Meinardus, Richmond, RothSzekeres}, to name a few.

\subsection{Statement of the results}

Generalizing the infinite product which defines $\eta,$ we consider the ubiquitous $q$-infinite products
\begin{equation*}
F_1(\xi; q):=\prod_{n=1}^{\infty}\left(1-\xi q^n\right),\ \ \
F_2(\xi; q):=\prod_{n=1}^{\infty}\left(1-(\xi q)^n\right),
\ \ \ {\text {\rm and}}\ \ \ F_3(\xi;q):=\prod_{n=1}^{\infty} \left(1-\xi^{-1}(\xi q)^n\right).
\end{equation*}
These infinite products are common as factors of generating functions in combinatorics, number theory, and representation theory.
We  obtain the asymptotic properties for $F_1(\xi;q), F_2(\xi;q),$ and $F_3(\xi;q),$ where $\xi$ is a root of unity, which are generally required for implementing
the Circle Method to such generating functions. This result is of independent interest.

To make this precise, we recall \textit{Lerch's transcendent}
\begin{align*}
\Phi(z,s,a):=\sum_{n=0}^\infty \frac{z^n}{(n+a)^s}.
\end{align*}
 Moreover, for coprime $h,k\in\N$ we define
\begin{align}\label{eqn: defn omega_h,k}
\omega_{h,k}:=\exp(\pi i \cdot s(h,k)),
\end{align}
using the \emph{Dedekind sum}
\begin{align*}
s(h,k):=\sum_{\mu \pmod k} \left(\left(\frac{\mu}{k}\right)\right)\left(\left(\frac{h\mu}{k}\right)\right).
\end{align*}
Here we use the standard notation
\begin{align*}
((x)):=\begin{cases} x-\lfloor x \rfloor-\frac{1}{2} & \text{if} \ x\in \mathbb R \setminus \mathbb Z, \\ 0 & \text{if} \ x \in \mathbb Z. \end{cases}
\end{align*}
For arbitrary positive integers $m$ and $n$, we define $\omega_{m,n} := \omega_{\frac{m}{\gcd(m,n)}, \frac{n}{\gcd(m,n)}}$. Note that $s(h,k)$ only depends on $h\pmod{k}$ and that $s(0,1)=0$. Moreover, we let
\begin{align} \label{LambdaEQ} 
\lambda_{t,a,b,h,k} := \gcd(k,t)  \begin{cases} 1 & \text{if } k=1 \text{ or } \lp k > 1 \text{ and } b \centernot | \frac{k}{\gcd(k,t)}\rp, \\ b & \text{if } b | \frac{k}{\gcd(k,t)} \text{ and } \frac{ht}{\gcd(k,t)} + a\frac{k}{b \gcd(k,t)} \not \equiv 0 \pmod{b}, \\ b^2 & \text{if } b | \frac{k}{\gcd(k,t)} \text{ and } \frac{ht}{\gcd(k,t)} +a \frac{k}{b \gcd(k,t)} \equiv 0 \pmod{b}. \end{cases}
\end{align}

 For $0\leq \theta < \frac{\pi}{2}$, we define the domain 
 \begin{align}\label{eqn: defn D_theta}
 	D_{\theta} \coloneqq \left\{ z=re^{i\alpha} \colon r \geq 0 \text{ and } |\alpha| \leq \theta \right\}.
 \end{align}

\begin{theorem}\label{Theorem1} 
Assume the notation above. For $b>0$, let $\xi$ be a primitive $b$-th root of unity, then the following are true. 
\begin{enumerate}[leftmargin=*]
	\item[\rm (1)] As $z \to 0$ in $D_\theta$ we have 
	\begin{align*}
	F_{1}\left(\xi;e^{-z}\right)  =\frac{1}{\sqrt{1-\xi}} \, e^{-\frac{\xi\Phi(\xi,2,1)}{z}}\lp 1+O\left(|z|\right) \rp.
	\end{align*}
	
	\item[\rm (2)] Suppose that $b$ is an odd prime, and let $\xi = e^{\frac{2\pi i a}{b}}$, $t \in \N$, $q = e^{\frac{2\pi i}{k}(h + iz)}$ for $0 \leq h < k$ with $\gcd(h,k) = 1$, and $z \in \C$ with $\mathrm{Re}(z) > 0$. 
	Then as $z \to 0$ we have 
$$F_2\left(\xi;q^t\right) \sim \omega_{\frac{hbt+ak}{\lambda_{t,a,b,h,k}}, \frac{kb}{\lambda_{t,a,b,h,k}}}^{-1} \left(\frac{\lambda_{t,a,b,h,k}}{tbz}\right)^{\frac 12} e^{-\frac{\pi \lambda_{t,a,b,h,k}^2}{12 b^2 ktz}}.$$

\item[\rm (3)] As $z\to 0$ in $D_\theta$, we have
\begin{align*}
F_3\left(\xi;e^{-z}\right)= \frac{\sqrt{2\pi} \left(b^2z\right)^{\frac 12-\frac 1b}}{\Gamma\left(\frac{1}{b}\right)}
\prod_{j=1}^{b-1}\frac{1}{(1-\xi^j)^{\frac jb}}
e^{-\frac{\pi^2}{6b^2z}}\lp 1+  O\left(|z|\right) \rp.
\end{align*} 
\end{enumerate}
\end{theorem}

\begin{remark}
If $\xi=1$ and $q=e^{2\pi i \tau}$, then we have $$F_1(1;q)=F_2(1;q)=F_3(1;q)=q^{-\frac{1}{24}}\eta(\tau).$$
Asymptotic properties in this case are well-known consequences of the modularity of $\eta(\tau).$
\end{remark}

\subsection{The Euler--Maclaurin summation formula}\label{sec:eulermaclaurinsummationformula}
We require the following generalization of the Euler--Maclaurin summation formula. To state it, we need some notation.  For $s,z\in\C$ with $\operatorname{Re}(s)>1, \operatorname{Re}(z)>0$, we recall the {\it Hurwitz zeta function} $\zeta(s,z):=\sum_{n=0}^\infty \frac{1}{(n+z)^s},$ the \emph{digamma function} $\psi(x):=\frac{\Gamma'(x)}{\Gamma(x)},$ and the Euler--Mascheroni constant $\gamma$. Furthermore, we let $B_n(x)$ denote the \emph{$n$-th Bernoulli polynomial} defined via its generating function $\frac{te^{xt}}{e^t-1}=\sum_{n=0}^\infty B_n(x)\frac{t^n}{n!}$. The consequence of the Euler--Maclaurin summation formula required is described by the following lemma. A function $f$ on a domain in $\C$ is of \textit{sufficient decay} if there exists $\varepsilon >0$ such that $f(w) \ll w^{-1-\varepsilon}$ as $|w| \rightarrow \infty$ in the domain. Throughout we say that
\[
	f(z) \sim \sum_{n=0}^\infty a_nz^n
\]
if for any $N\in\N_0$, $f(z)=\sum_{n=0}^N a_nz^n+O(|z|^{N+1})$.

\begin{lemma}\label{lemma11}
	Let $0 < a \leq 1$ and $A \in \R^+$, and let $D_{\theta}$ be defined by \eqref{eqn: defn D_theta}. Assume that $f(z) \sim \sum_{n=n_0}^{\infty} c_n z^n$ $(n_0\in\Z)$ as $z \rightarrow 0$ in $D_\theta$. Furthermore, assume that $f$ and all of its derivatives are of sufficient decay in $D_\theta$ in the above sense. Then we have that
	\begin{align*}
	\sum_{n=0}^\infty f((n+a)z)\sim \sum_{n=n_0}^{-2} c_{n} \zeta(-n,a)z^{n}+ \frac{I_{f,A}^*}{z}-\frac{c_{-1}}{z} \left( \Log \left(Az \right) +\psi(a)+\gamma \right)-\sum_{n=0}^\infty c_n \frac{B_{n+1}(a)}{n+1} z^n,
	\end{align*}
	as $z \rightarrow 0$ uniformly in $D_\theta$, where 
		\begin{align*}
			I_{f,A}^*:=\int_{0}^{\infty} \left(f(u)-\sum_{n=n_0}^{-2}c_{n}u^n-\frac{c_{-1}e^{-Au}}{u}\right)du.
		\end{align*}
\end{lemma}
\begin{remark}
	Note that for $a=1$, we have that $\psi(a)+\gamma=0$.
\end{remark}
\begin{proof}[Proof of Lemma \ref{lemma11}]
	A generalization of an observation of Zagier \cite[Proposition 3]{Za} is that of \cite[Theorem 1.2]{BJM}, which states the following. Let $h$ be a holomorphic function on a domain containing $D_\theta$, so that in particular $h$ is holomorphic at the origin, such that $h$ and all of its derivatives have sufficient decay, and $h(z) \sim \sum_{n=0}^{\infty} c_n z^n$ as $z \rightarrow 0$ in $D_\theta$.  Furthermore, let $I_h \coloneqq \int_0^\infty h(x)dx$. Then we have for $a\in\R$
	\begin{align}\label{Eqn: EM holomorphic}
	\sum_{n=0}^\infty h((n+a)z)\sim\frac{I_h}{z}-\sum_{n=0}^\infty c_n \frac{ B_{n+1}(a)}{n+1}z^n,
	\end{align}
	as $z \rightarrow 0$ in $D_\theta$. For the given $A$, write
	\begin{align}\label{Eqn: f as g}
	f(z) = g(z) + \frac{c_{-1}e^{-Az}}{z} + \sum_{n=n_0}^{-2} c_nz^n,
	\end{align}
	which means that
	\begin{align*}
	g(z) = f(z) - \frac{c_{-1}e^{-Az}}{z} - \sum_{n=n_0}^{-2} c_nz^n.
	\end{align*}
	The final term in \eqref{Eqn: f as g} yields the first term in the right-hand side of the lemma. Since $g$ has no pole, \eqref{Eqn: EM holomorphic} gives that
	\begin{align*}
	\sum_{n=0}^\infty g((n+a)z)\sim\frac{I_g}{z}- \sum_{n=0}^\infty c_n(g) \frac{ B_{n+1}(a)}{n+1}z^n,
	\end{align*}
	where $c_n(g)$ are the coefficients of $g$. Note that $I_g = I_{f,A}^*$. We compute that
	\begin{align*}
	- \sum_{n=0}^\infty c_n(g) \frac{ B_{n+1}(a)}{n+1}z^n = -  \sum_{n=0}^\infty \left(c_n - \frac{(-A)^{n+1} c_{-1}}{(n+1)!}\right) \frac{ B_{n+1}(a)}{n+1}z^n.
	\end{align*}
	Combining the contribution from the second term with the contribution from the second term from \eqref{Eqn: f as g}, we obtain
	\begin{align*}
	\frac{c_{-1}}{z} \left( \sum_{n =0}^{\infty} \frac{e^{-A(n+a)z}}{n+a} + \sum_{n=1}^{\infty} \frac{B_n(a)}{n \cdot n!} (-Az)^n \right).
	\end{align*}
Using \cite[equation (5.10)]{BJM}, the term in the paranthesis is equal to $-(\Log(Az)+\psi(a)+\gamma)$. Combining the contributions yields the statement of the lemma.
\end{proof}

\subsection{An integral evaluation}

We require the following integral evaluation.
\begin{lemma}\label{lem:int}
We have for $N\in\R^+$
\begin{multline*}
\int_0^\infty\left(\frac{e^{-x}}{x\left(1-e^{Nx}\right)}-\frac{1}{Nx^2}+\left(\frac 1N-\frac 12\right)\frac{e^{-x}}{x} \right)dx
\\=\log\left(\Gamma\left(\frac 1N\right) \right) +\left(\frac 12-\frac 1N\right) \log\left(\frac 1N\right)-\frac 12\log(2\pi). 
\end{multline*}
\end{lemma}

\begin{proof}
	Making the change of variables $x\mapsto\frac xN$, the left-hand side equals
	\begin{equation*}
	\int_0^\infty\left(\frac{e^{-\frac{x}{N}}}{x\left(1-e^{-x}\right)} - \frac{1}{x^2} +\left(\frac 1N-\frac 12\right)\frac{e^{-\frac {1}N}}{x}  \right) dx.
	\end{equation*}
	Now write
	\begin{equation*}
	\frac{1}{x \left(1-e^{-x} \right)}=\frac 1x +\frac{1}{x \left(e^x-1 \right)}.
	\end{equation*}
	Thus the integral becomes
	\begin{equation*}
	\int_0^\infty\left(\frac{1}{e^x-1}+\frac 12-\frac 1x\right)\frac{e^{-\frac{x}{N}}}{x}dx
	+\int_0^\infty\left(\frac{e^{-\frac{x}{N}}}{x} - \frac{1}{x^2} +\left(\frac 1N-\frac 12\right)\frac{e^{-\frac {x}N}}{x}  -\frac{e^{-\frac{x}{N}}}{2x}+\frac{e^{-\frac{x}{N}}}{x^2}\right) dx.
	\end{equation*}
	We evaluate the second integral as $-\frac{1}{N}$. The claim now follows, using Binet's first integral formula (see 12.31 of \cite{WW}).
\end{proof}

\subsection{Proof of Theorem~\ref{Theorem1}}

We employ the generalized Euler--Maclaurin summation formula to prove Theorem~\ref{Theorem1} (1) and (3); for part (2) we use modularity.

\subsubsection{Proof of Theorem~\ref{Theorem1}~\normalfont{(1)}}

 Let $|z|<1$. Taking logarithms, we have 
 \begin{align*}
 G_{\xi}\left(e^{-z}\right)&:=\operatorname{Log} \left(F_{1}\left(\xi;e^{-z}\right) \right)
 =-z\sum_{j=1}^b \xi^j \sum_{m=0}^{\infty}f\left(\left(m+\frac j b\right)bz\right),
 \end{align*}
 where
 \begin{equation*}
 f(z):=\frac{e^{-z}}{z\left(1-e^{-z}\right)}=\frac{1}{z^2}-\frac{1}{2z} +\sum_{n=0}^\infty \frac{B_{n+2}}{(n+2)!}z^n.
 \end{equation*}
 By Lemma \ref{lemma11}, it follows that
 \begin{align*}
 \sum_{m=0}^\infty f \left( \left( m+\frac{j}{b} \right)bz\right) =\frac{\zeta\left(2,\frac j b\right)}{b^2z^2}+\frac{I_{f,1}^*}{bz} +{\frac{1}{2bz}}\left(\Log \left( {bz}\right) +\psi \left(\frac{j}{b}\right)+\gamma \right)
 +O(1).
 \end{align*}
 Therefore, we find that
 \[
 G_{\xi} \left(e^{-z}\right) = -\frac{1}{b^2z} \sum_{j=1}^b \xi^j \zeta\left(2,\frac jb\right) -\frac{I_{f,1}^*}{b} \sum_{j=1}^b \xi^{j}-\frac{1}{2b} \sum_{j=1}^b \xi^j \left(\Log\left(bz\right) +\psi\left(\frac jb\right)+\gamma\right) +O(|z|).
 \]
Now note that $\sum_{j=1}^b \xi^{j}=0.$
 Moreover, we require the identity \cite[p. 39]{Campbell}  (correcting a minus sign and erroneous $k$ on the right-hand side)
 \begin{align}\label{digam}
 \sum_{j=1}^b \psi \left( \frac{j}{b} \right) \xi^{j}=b\operatorname{Log }\left( 1-\xi \right).
 \end{align}
 Combining these observations, we obtain
 \begin{align*}
 G_{\xi}\left( e^{-z} \right)=-\frac{1}{b^2z} \sum_{j=1}^b \xi^{j} \zeta\left( 2,\frac{j}{b} \right)-\frac{1}{2} \operatorname{Log }(1-\xi) + O\left(|z|\right).
 \end{align*}
 After noting that
 \begin{align*}
 \sum_{j=1}^b \xi^{j} \zeta \left(2,\frac{j}{b} \right)&=b^2 \xi\Phi(\xi,2,1),
 \end{align*}
 the claim follows by exponentiation. \qed  \\

\subsubsection{Proof of Theorem~\ref{Theorem1}~\normalfont{(2)}}
Note that 
\[
F_2\left(\xi;q^t\right) = \left(\xi q^t; \xi q^t\right)_\infty,
\]
where $(q;q)_\infty := \prod_{j=1}^\infty (1 - q^j)$. The classical modular transformation law for the Dedekind $\eta$-function  (see 5.8.1 of \cite{CohenStromberg}) along with the identity $\eta(\tau) = q^{ \frac{1}{24}} (q;q)_\infty$ implies that
\begin{align}\label{Eqn: transform usual pochham}
(q;q)_\infty = \omega_{h,k}^{-1} z^{-\frac 12} e^{\frac{\pi}{12k}\left( z - \frac{1}{z} \right)} (q_1;q_1)_\infty,
\end{align}
where $q_1 := e^{\frac{2\pi i}{k}( h' + \frac{i}{z})}$ where $0 \leq h' < k$ is defined by $h h' \equiv -1 \pmod{k}$ and $\omega_{h,k}$ is defined as in \eqref{eqn: defn omega_h,k}. In particular, this implies that
\begin{equation} \label{Eta asymptotic}
(q;q)_\infty \sim \omega_{h,k}^{-1} z^{-\frac 12} e^{-\frac{\pi}{12kz}}
\end{equation}
as $z\rightarrow0$ with $\re{z}>0$.
Now, by using the definitions of $\xi, q$ given in the statement of Theorem 2.1 (2) we have
\[
\xi q^t = e^{\frac{2\pi i}{kb}\left( hbt + ak + itbz\right)}.
\]
We claim that $\lambda_{t,a,b,h,k}$ as defined in \eqref{LambdaEQ} satisfies $\lambda_{t,a,b,h,k} = \gcd(kb, hbt + ak)$. If $k=1$, then the claim is clear, and so we assume that $k > 1$. Write $k = \gcd(k,t) k_1$ and $t = \gcd(k,t) t_1$. Then we have
\[
\gcd(kb, hbt + ak) = \gcd(k,t) \gcd(k_1 b, hbt_1 + ak_1).
\]
Noting that $\gcd(k_1,b)$ divides each of $k_1b, hbt_1$, and $ak_1$, it follows that
\[
\gcd(kb, hbt + ak) = \gcd(k,t) \gcd(k_1, b) \gcd\left( \frac{k_1 b}{\gcd(k_1, b)}, \frac{hbt_1}{\gcd(k_1, b)} + a\frac{k_1}{\gcd(k_1,b)} \right).
\]
Note that, since $b$ is prime, $\gcd(k_1, b) \in \{ 1, b \}$. If $\gcd(k_1,b) = 1$, then
\[
\gcd(k_1 b, hbt_1 + ak_1) = \gcd(k_1, hbt_1) \gcd(b, ak_1) = 1.
\]

If on the other hand $\gcd(k_1,b) = b$, then write $k_1 = b^\kappa k_2$ with $\gcd(k_2, b) = 1$. Then
\begin{align*}
\gcd\left(k_1, ht_1 + a\frac{k_1}{b}\right) &= \gcd\left(b^\kappa k_2, ht_1 + a k_2 b^{\kappa-1}\right)
= \gcd\left(b^\kappa, ht_1 + ak_2 b^{\kappa-1}\right) \gcd(k_2, ht_1)\\
&= \gcd\left(b^\kappa, ht_1 + a k_2 b^{\kappa-1}\right).
\end{align*}
If $\kappa > 1$, then $\gcd(b^\kappa, ht_1 + ak_2 b^{\kappa-1}) = 1$ since $\gcd(b, ht_1) = 1$. If $\kappa = 1$, then we are left with $\gcd(b, ht_1 + ak_2)$. Therefore, we obtain
\[
\gcd(kb, hbt + ak) = \gcd(k,t) \begin{cases} 1 & \text{if } b \centernot | \frac{k}{\gcd(k,t)}, \\ b & \text{if } b | \frac{k}{\gcd(k,t)} \text{ and } \frac{ht}{\gcd(k,t)} + a\frac{k}{b \gcd(k,t)} \not \equiv 0 \pmod{b}, \\ b^2 & \text{if } b | \frac{k}{\gcd(k,t)} \text{ and } \frac{ht}{\gcd(k,t)} + a\frac{k}{b \gcd(k,t)} \equiv 0 \pmod{b}, \end{cases}
\]
which is equal to $\lambda_{t,a,b,h,k}$.

It follows that $\gcd(\frac{kb}{\lambda_{t,a,b,h,k}},\frac{hbt+ak}{\lambda_{t,a,b,h,k}}) = 1$. Therefore, by making the replacements $h \mapsto \frac{hbt+ak}{\lambda_{t,a,b,h,k}}$, $k \mapsto \frac{kb}{\lambda_{t,a,b,h,k}}$, and $z \mapsto \frac{tbz}{\lambda_{t,a,b,h,k}}$ in \eqref{Eta asymptotic}, the result follows.\qed

\subsubsection{Proof of Theorem \ref{Theorem1}~\normalfont{(3)}}
 Again assume that $|z|<1$. Writing
 \begin{equation*}
 F_3(\xi;q)=\prod_{j=1}^b\prod_{n=0}^{\infty}\left(1-\xi^{j-1}q^{bn+j}\right),
 \end{equation*}
 we compute
 \begin{align*}
 \operatorname{Log}\left(F_3\left(\xi;e^{-z}\right)\right)=-z\sum_{1\leq j,r \leq b} \xi^{(j-1)r}\sum_{m=0}^\infty f_j\left(\left( m+\frac{r}{b}\right) bz\right),
 \end{align*}
 where $f_j(z):=\frac{e^{-jz}}{z(1-e^{-bz})}$.
 By Lemma \ref{lemma11}, we obtain 
 \begin{equation*}\label{eqn:obtain}
 \sum_{m=0}^{\infty}f_j\left(\left(m+\frac r b\right)bz \right)
 \sim
 \frac{\zeta\left(2,\frac r b\right)}{b^3z^2}+
 \frac{I_{f_{j,1}}^*}{bz}+\frac{B_1\left(\frac{j}{b}\right)}{bz}\left(\Log\left({bz}\right)+\psi \left(\frac{r}{b}\right)+\gamma\right)+O(1)
 \end{equation*}
 The first term contributes $-\frac{\pi^2}{6b^2z}$.
 By Lemma \ref{lem:int}, the second term contributes 
 \begin{align*}
 -\frac 1b \sum_{j=1}^{b}I_{f_{j,1}}^*\sum_{r=1}^{b}\xi^{(j-1)r}& =-I_{f_{1,1}}^*
  =-\log\left(\Gamma\left(\frac{1}{b}\right)\right) - \left(\frac{1}{2}-\frac{1}{b}\right)\log\left(\frac{1}{b}\right)+\frac 12\log(2\pi) \\
 &= \log\left( \frac{b^{\frac{1}{2} -\frac{1}{b}} (2\pi)^{\frac{1}{2}} }{\Gamma\left(\frac{1}{b}\right)} \right).
 \end{align*}
  Next we evaluate
 \begin{align*}
 -\frac{1}{b}\left(\Log\left({bz}\right)+\gamma\right)\sum_{1\leq j\leq b}B_1\left(\frac{j}{b}\right)  \sum_{1\leq r\leq b} \xi^{(j-1)r}=-B_1\left(\frac{1}{b}\right)\left(\Log\left({bz}\right)+\gamma\right).
 \end{align*}
 \indent Finally we are left to compute
 \begin{align*}
 -\frac{1}{b}\sum_{1\leq j,r \leq b} \xi^{(j-1)r} \left(\frac{j}{b}-\frac{1}{2}\right)\psi\left(\frac{r}{b}\right)=-\frac{1}{b}\sum_{\substack{0\leq j\leq b-1 \\ 1\leq r \leq b}} \xi^{jr}\left(\frac{j}{b}+\frac{1}{b}-\frac{1}{2}\right)\psi\left(\frac{r}{b}\right).
 \end{align*}
 The $(\frac{1}{b}-\frac{1}{2})$-term yields $\gamma(\frac{1}{b}-\frac{1}{2})$. Thanks to \eqref{digam}, the $\frac{j}{b}$ term contributes
 \begin{align*}
 -\frac{1}{b^2} \sum_{0\leq j\leq b-1} j \sum_{1\leq r \leq b} \psi \left(\frac{r}{b}\right) \xi^{jr}=-\frac{1}{b}\sum_{1\leq j\leq b-1}j\operatorname{Log }\left(1-\xi^j\right).
 \end{align*}
 Combining these observations yields that
 \begin{align*}
 	\operatorname{Log}\left(F_3\left(\xi;e^{-z}\right)\right) = \log\left( \frac{b^{\frac{1}{2} -\frac{1}{b}} (2\pi)^{\frac{1}{2}} }{\Gamma\left(\frac{1}{b}\right)} \right) -\frac{\pi^2}{6b^2z} -B_1\left(\frac{1}{b}\right)\Log\left({bz}\right)  - \sum_{1\leq j\leq b-1} \frac{j}{b} \operatorname{Log }\left(1-\xi^j\right) + O\left(|z|\right).
 \end{align*}
Exponentiating gives the desired claim. \qed

\section{Proof of Theorem~\ref{Theorem2},  Corollary~\ref{Corollary3}, and Theorem ~\ref{Vanishing}} \label{Section3}

Here we recall a beautiful $q$-series identity of Han, who offered  the generating functions we require for Theorems~\ref{Theorem2} and \ref{Vanishing}, and Corollary~\ref{Corollary3}. Apart from factors which naturally correspond to
quotients of Dedekind's eta-function, these generating functions have factors of the form
$F_2(\xi;q^t)^{-t}$. 
The proof of Theorem~\ref{Vanishing} follows directly from this fact along with known identities for
the $2$-core and $3$-core generating functions.
To prove Theorem~\ref{Theorem2}, we apply Zuckerman's exact formulas to these functions, making strong use of Theorem~\ref{Theorem1} (2).

\subsection{Work of Han}  Here we derive the generating functions for the modular $t$-hook functions
$p_t(a,b;n)$. To this end, we recall the following beautiful formula of Han that he
derived in his work on extensions of the celebrated Nekrasov--Okounkov formula\footnote{This formula was also obtained by Westbury (see Proposition 6.1 and 6.2 of \cite{Westbury}).} (see (6.12) of \cite{NekrasovOkounkov}) with $w \in \C$:
$$
\sum_{\lambda \in \mathcal{P}} q^{|\lambda|} \prod_{h\in \mathcal{H}(\lambda)} \left(1-\frac{w}{h^2}\right)
=\prod_{n=1}^{\infty}\left(1-q^n\right)^{w-1}.
$$
Here $\mathcal{P}$ denotes the set of all integer partitions, including the empty partition, and $\mathcal{H}(\lambda)$ denotes the multiset of hook lengths for $\lambda.$
Han \cite{Han} proved the following beautiful identity for the generating function for $t$-hooks in partitions
\begin{equation*}
H_t(\xi;q):=\sum_{\lambda \in \mathcal{P}} \xi^{\# \mathcal{H}_t(\lambda)}q^{|\lambda|}.
\end{equation*}

\begin{theorem}{\text {\rm (Corollary 5.1 of \cite{Han})}}\label{HanFunction}
As formal power series, we have
$$
H_t(\xi;q)=\frac{1}{F_2(\xi;q^t)^t}\prod_{n=1}^{\infty}
\frac{\left(1-q^{tn}\right)^t}{1-q^n}.
$$
\end{theorem}

As a corollary, we obtain the following generating function for $p_t(a,b;n).$
\begin{corollary}\label{ptabGenFunctions}
If $t>1$ and $0\leq a<b$, then as formal power series we have
\begin{equation*}\label{Orthogonality}
H_t(a,b;q):=\sum_{n=0}^{\infty}p_t(a,b;n)q^n=\frac{1}{b} \sum_{r=0}^{b-1} \zeta_b^{-ar}H_t\left(\zeta_b^r;q\right),
\end{equation*}
where $\zeta_b:=e^{\frac{2\pi i}b}.$
\end{corollary}
\begin{proof}
We have that
\begin{displaymath}
\begin{split}
\frac{1}{b} \sum_{r=0}^{b-1} \zeta_b^{-ar} H_t(\zeta_b^r;q)&=
\frac{1}{b} \sum_{\lambda \in \mathcal{P}}q^{|\lambda|} \sum_{r=0}^{b-1}\zeta_b^{\left(\#\mathcal{H}_t(\lambda)-a\right)r}=H_t(a,b;q). \qedhere
\end{split}
\end{displaymath}
\end{proof}

The dependence of $H_t(\xi;q)$ on $F_2(\xi;q^t)$ enables us to compute asymptotic behavior of $H_t(\xi;q)$ using Theorem \ref{Theorem1} (2) and, by Corollary \ref{ptabGenFunctions}, the asymptotic behavior of $H_t(a,b;q)$.

\subsection{Proof of Theorem~\ref{Vanishing}}
Here we prove Theorem~\ref{Vanishing}. We first consider the case (1), where $\ell$ is an odd prime. We consider the generating function, using Corollary \ref{ptabGenFunctions}
$$
H_2(a_1,\ell ;q)=\sum_{n=0}^{\infty}p_2(a_1,\ell;n)q^n=\frac{1}{\ell} \sum_{r_1=0}^{\ell-1} \zeta_\ell^{-a_1 r_1}H_2\left(\zeta_\ell^{r_1};q\right).
$$
Applying again orthogonality of roots of unity, keeping only those terms $a_2\pmod \ell$,  where
$a_2\in \{0, 1,\dots, \ell-1\}$, we find that
$$
\sum_{n=0}^{\infty} p_2(a_1,\ell;\ell n+a_2)q^{\ell n +a_2}=
\frac{1}{\ell^2}\sum_{r_1, r_2\pmod \ell}\zeta_{\ell}^{-a_1 r_1 -a_2 r_2} H_2\left(\zeta_{\ell}^{r_1};\zeta_{\ell}^{r_2}q\right).
$$
Making use of the definition of $H_t(\xi;q)$, 
if we define $\mathcal{B}_2(q)$ and $\mathcal{C}_2(q)$ by
\begin{equation}\label{qidentities}
\mathcal{B}_2(q)=\sum_{n=0}^{\infty}b_2(n)q^n:=\prod_{n=1}^{\infty}\frac{1}{\left(1-q^n\right)^2} \ \ \ \
{\text {\rm and}}\ \ \ \ 
\mathcal{C}_2(q):=\prod_{n=1}^{\infty}\frac{\left(1-q^{2n}\right)^2}{1-q^n},
\end{equation}
then we have 
$$
\sum_{\substack{n\geq 0 \\n\equiv a_2\pmod \ell}}p_2(a_1,\ell;n)q^n=
\frac{1}{\ell^2}\sum_{r_1, r_2\pmod \ell}\zeta_{\ell}^{-a_1 r_1 -a_2 r_2}
\mathcal{B}_2\left(\zeta_{\ell}^{r_1+2r_2}q^2\right) \mathcal{C}_2\left(\zeta_{\ell}^{r_2} q\right).
$$
Thanks to the classical identity of Jacobi
\begin{equation*}
\mathcal{C}_2(q)=\sum_{k=0}^{\infty}q^{\frac{k(k+1)}{2}},
\end{equation*}
 for $N\equiv a_2\pmod{\ell}$, we find that 
\begin{align}\label{key}
p_2(a_1,\ell;N)&=\frac{1}{\ell^2}\sum_{r_1, r_2\pmod \ell}\zeta_{\ell}^{-a_1 r_1-a_2 r_2}
\sum_{\substack{k,m\geq 0\\
2m+\frac{k(k+1)}{2}=N}} b_2(m)\zeta_{\ell}^{(r_1+2r_2)m+r_2\frac{k(k+1)}{2}}
\nonumber \\
&=\sum_{\substack{m\equiv a_1\pmod \ell\\
		2m+\frac{k(k+1)}{2}=N}} b_2(m),
\end{align}
by making the linear change of variables $r_1\mapsto r_1-2r_2$ and again using orthogonality of roots of unity.
This then requires the solvability of the congruence $a_2-2a_1\equiv \frac{k(k+1)}{2}\pmod \ell.$
Completing the square produces the quadratic residue condition which prohibits this solvability, and hence
completes the proof of (1).

The proof of (2) follows similarly, with $\ell$ replaced by $\ell^2$ for primes $\ell\equiv 2\pmod 3.$ The functions in (\ref{qidentities})
are replaced with
$$
\mathcal{B}_3(q)=\sum_{n=0}^{\infty}b_3(n)q^n:=\prod_{n=1}^{\infty}\frac{1}{\left(1-q^n\right)^3} \ \ \ \ 
{\text {\rm and}}\ \ \ \ \mathcal{C}_3(q):=\prod_{n=1}^{\infty}\frac{\left(1-q^{3n}\right)^3}{1-q^n}.
$$
It is well-known that (for example, see Section 3 of \cite{GranvilleOno} or \cite[Lemma 2.5]{HanOno}),
$$
\mathcal{C}_3(q)=:\sum_{n=0}^{\infty}c_3(n)q^n=\sum_{n=0}^{\infty}\sum_{d\mid (3n+1)}\legendre{d}{3}q^n.
$$
For primes $\ell\equiv 2\pmod 3$, this implies that  $c_3(\ell^2 n+a)=0$ for every positive integer $n$, whenever $\ord_{\ell}(3a+1)=1$.
For example, this means that $c_3(4n+3)=0$ if $\ell=2$.

Let $0\leq a_1, a_2<\ell^2$. In direct analog with (\ref{key}), a calculation reveals that nonvanishing for $N\equiv a_2\pmod {\ell^2}$ relies  on sums of the form
$$
\sum_{\substack{m\equiv a_1\pmod{\ell^2}\\
3m+k=N}} b_3(m)c_3(k).
$$
If  $\ord_{\ell}(3a+1)=1$ and
$a_2-3a_1\equiv a\pmod{\ell^2}$, then $p_3(a_1,\ell^2; \ell^2+a)=0.$
 This is claim (2).

\subsection{Evaluating certain Kloosterman sums}

The proof of Theorem \ref{Theorem2} relies on the arithmetic of the Kloosterman sums
\begin{equation*}\label{KloostermanSumDfn}
	K(a,b,t;n) := \sum_{h=1}^{b-1} \frac{\omega_{h,b}}{\omega_{th,b}^t} \zeta_b^{(at-n)h},
\end{equation*}
where $b$ is an odd prime, and $s \geq 1$, $t > 1$ are integers. We evaluate this sum if $t$ is coprime to $b$. We start by computing $\omega_{h,b} \omega_{th,b}^{-t}$.

\begin{prop} \label{Dedekind Simplification}
Let $b$ be an odd prime, $h$, $t$ integers coprime to $b$, and let $\omega_{h,k}$ be defined by \eqref{eqn: defn omega_h,k}. Then we have
\begin{align*}
\frac{\omega_{h,b}}{\omega_{th,b}^t} = \lp \frac{h}{b} \rp \lp \frac{th}{b} \rp^t e^{\pi i \frac{(1-t)(b-1)}{4}} e^{\frac{2\pi i}{b} \frac{1}{24} \left(1-t^2\right)\left(1-b^2\right)h}.
\end{align*}

\end{prop}

\begin{proof}
The proof of this proposition uses the $\eta$-multiplier, which we label $\psi$. Theorem 5.8.1 of \cite{CohenStromberg} yields that
for $\lp \begin{smallmatrix} \alpha & \beta  \\ \gamma  & \delta  \end{smallmatrix} \rp \in \text{SL}_2(\Z)$  with $\gamma  > 0$ odd, we have $$\psi \begin{pmatrix} \alpha & \beta  \\ \gamma  & \delta  \end{pmatrix} = \left( \frac{\delta }{\gamma } \right) e^{\frac{\pi i}{12} \left( (\alpha +\delta )\gamma  - \beta \delta \left(\gamma ^2-1\right) - 3\gamma  \right)}.$$
We also have from formula (57b) of \cite{RademacherGrosswald} that for $\lp\begin{smallmatrix} \alpha & \beta  \\ \gamma  & \delta  \end{smallmatrix}\rp \in \text{SL}_2(\Z)$
$$
\psi \begin{pmatrix} \alpha & \beta  \\ \gamma  & \delta  \end{pmatrix} = e^{\pi i \left( \frac{\alpha+\delta }{12\gamma } - \frac{1}{4} \right)}\omega_{\delta,\gamma}^{-1} .
$$
By letting $\delta  = h$, $\gamma  = b,$ we obtain
\begin{align*}
\omega_{h,b} = \lp \frac{h}{b} \rp e^{\pi i \lp \frac{1}{12b}(\alpha+h - \beta hb)\left(1-b^2\right) + \frac{b-1}{4} \rp},
 \end{align*}
 where $\alpha,\beta $ satisfy $\alpha h-\beta b=1$. We therefore may conclude that
\begin{align*}
\frac{\omega_{h,b}}{\omega_{th,b}^t} = \lp \frac{h}{b} \rp \lp \frac{th}{b} \rp^t e^{\pi i \frac{(1-t)(b-1)}{4}} e^{\frac{\pi i}{12b} \lp (\alpha - tA )\left(1-b^2\right) + h\left(1 - \beta b - t^2\left(1-B b\right)\right)\left(1-b^2\right) \rp},
\end{align*}
where $\alpha h - \beta b = A  th - B  b = 1$. A straightforward calculation then gives the claim.
\end{proof}

We now turn to evaluating the Kloosterman sum $K(a,b,t;n)$. 

\begin{prop} \label{Kloosterman Sum Evaluation}
Suppose that $b$ is an odd prime, $a, n$ are integers, and $t > 1$ is an integer coprime to $b$. Then we have

\[ 
K(a,b,t;n) =
\begin{cases}
\mathbb{I}(a,b,t,n) (-1)^{\frac{(1-t)(b-1)}{4}} \left( \frac{t}{b} \right) & \text{ if } t \text{ is odd}, \vspace{5pt} \\  (-1)^{\frac{(1-t)(b-1)}{4}}\varepsilon_b \left( \frac{\frac{1}{24}\left(1-t^2\right)\left(1-b^2\right) + at - n}{b} \right)  \sqrt{b} & \text{ if } t \text{ is even,}
\end{cases}
\]
where $\mathbb{I}(a,b,t,n)$ is defined by \eqref{eqn: defn of I}.
\end{prop}

\begin{proof}
By Proposition \ref{Dedekind Simplification}, we have
\begin{align*}
	K(a,b,t;n) &= e^{\frac{\pi i}{4}(1-t)(b-1)} \sum_{h=1}^{b-1} \left(\frac hb\right) \left(\frac{th}{b}\right)^t \zeta_b^{(at-n)h+\frac{1}{24}\left(1-t^2\right)\left(1-b^2\right)h}.
\end{align*}

The multiplicativity of the Legendre symbol implies
\begin{align*}
\lp \frac{h}{b} \rp \lp \frac{th}{b} \rp^t = \lp \frac{h}{b} \rp^{t+1} \lp \frac{t}{b} \rp^t =
\begin{cases}
\lp \frac{t}{b} \rp & \text{ if } t \text{ is odd}, \vspace{5pt} \\ \lp \frac{h}{b} \rp & \text{ if } t \text{ is even}.
\end{cases}
\end{align*}

We proceed distinguishing on the parity of $t$.
Suppose first that $t$ is odd. Then since $b$ is odd, $\frac14 (1-t)(b-1)$ is an integer and the claim directly follows.

Suppose next that $t$ is even. Then we have
\begin{align*}
	K(a,b,t;n) = e^{\pi i\frac{(1-t)(b-1)}{4}} \sum_{h=1}^{b-1} \left(\frac hb\right) \zeta_b^{h\left(\frac{1}{24}\left(1-t^2\right)\left(1-b^2\right)+at-n\right)}.
\end{align*}
Using the classical evaluation of the Gauss sum (see for example pages 12-13 of \cite{Davenport}), we obtain
\[
	\sum_{h=1}^{b-1} \left(\frac hb\right) \zeta_b^{\left(\frac{1}{24}\left(1-t^2\right)\left(1-b^2\right)+at-n\right)h} = \left(\frac{\frac{1}{24}\left(1-t^2\right)\left(1-b^2\right)+at-n}{b}\right) \varepsilon_b \sqrt{b}. \qedhere
\]
\end{proof}

\subsection{An exact formula of Zuckerman}\label{Sec: Zuc}
Here we recall a result of Zuckerman \cite{Zuckerman}, building on work of Rademacher \cite{RademacherExact}. Using the Circle Method, Zuckerman computed exact formulae for Fourier coefficients for weakly holomorphic modular forms of arbitrary non-positive weight on finite index subgroups of $\SL_2(\Z)$ in terms of the cusps of the underlying subgroup and the principal parts of the form at each cusp. 
Let $F$ be a weakly holomorphic modular form of weight $\kappa \leq 0$ with transformation law
$$F(\gamma \tau) = \chi(\gamma) (c \tau + d)^{\kappa} F(\tau),$$
for all $\gamma = \left( \begin{smallmatrix} a & b \\ c & d \end{smallmatrix} \right)$ in some finite index subgroup of $\SL_2(\Z)$. The transformation law can be viewed alternatively in terms of the cusp $\frac{h}{k} \in\Q$. Let $h'$ be defined through the congruence $hh' \equiv -1 \pmod k$. Taking $\tau = \frac{h'}{k} + \frac{i}{kz}$ and choose $\gamma=\gamma_{h,k}  := \left( \begin{smallmatrix} h & \beta \\ k & -h' \end{smallmatrix} \right) \in \textnormal{SL}_2(\mathbb{Z})$, we obtain the equivalent transformation law
$$
F\lp \frac{h}{k}+\frac{iz}{k} \rp = \chi(\gamma_{h,k})(-iz)^{-\kappa}   F\lp \frac{h'}{k}+\frac{i}{kz} \rp.
$$
Let $F$ have the Fourier expansion at $i\infty$ given by
\[
	F(\tau) = \sum_{n\gg-\infty} a(n)q^{n+\alpha}
\]
and Fourier expansions at each rational number $0 \leq \frac{h}{k} < 1$ given by
\begin{equation*}\label{2.5}
	F|_{\kappa}\gamma_{h,k}(\tau) = \sum_{n \gg -\infty} a_{h,k}(n) q^{\frac{n + \alpha_{h,k}}{c_{k}}}.
\end{equation*}
Furthermore, let $I_\alpha$ denote the usual $I$-Bessel function.
In this framework, the relevant theorem of Zuckerman \cite[Theorem 1]{Zuckerman} may be stated as follows.
\begin{theorem}\label{Thm: Zuckerman}
	Assume the notation and hypotheses above. If $n + \alpha > 0,$ then we have
	\begin{align*}
		&a(n) =  2\pi (n+\alpha)^{\frac{\kappa-1}{2}} \sum_{k=1}^\infty \dfrac{1}{k} \sum_{\substack{0 \leq h < k \\ \gcd(h,k) = 1}}\chi(\gamma_{h,k}) e^{- \frac{2\pi i (n+\alpha) h}{k}} 
		\\ 
		&\ \times \sum_{m+\alpha_{h,k} \leq 0} a_{h,k}(m) e^{ \frac{2\pi i}{k c_{k}} (m + \alpha_{h,k}) h' } \left( \dfrac{\lvert m +\alpha_{h,k} \rvert }{c_{k}} \right)^{ \frac{1 - \kappa}{2}} I_{-\kappa+1}\left( \dfrac{4\pi}{k} \sqrt{\dfrac{(n + \alpha)\lvert m +\alpha_{h,k} \rvert}{c_{k}}} \right).
	\end{align*}
\end{theorem}

\subsection{Proof of Theorem \ref{Theorem2} and Corollary \ref{Corollary3}}

We next provide proofs of both Theorem \ref{Theorem2} and Corollary \ref{Corollary3}. Our main tool is the powerful theorem of Zuckerman described in Section \ref{Sec: Zuc}.

\begin{proof}[Proof of Theorem \ref{Theorem2}]

 Using Corollary \ref{ptabGenFunctions} we have
\begin{align}\label{eqn: split}
H_t(a,b;q) = \frac{1}{b(q;q)_\infty} + \sum_{r=1}^{b-1}  \zeta_b^{-ar}H_t\left(\zeta_b^r;q\right).
\end{align}
From Theorem \ref{HanFunction} we conclude
\begin{align*}
H_t\left(\zeta_b^r;q \right) = \frac{\left(q^t;q^t\right)^t_\infty}{\left(\zeta_b^rq^t ; \zeta_b^rq^t\right)^t_\infty \left(q;q\right)_\infty}.
\end{align*}

 To obtain the transformation formula for $H_t(\zeta_b^r;q)$ at the cusp $\frac hk$, we write 
\begin{align*}
q^t = e^{\frac{2\pi i t}{k} \left( h+iz \right)} = e^{\frac{2\pi i }{\frac{k}{\gcd(k,t)}} \left( h\frac{t}{\gcd(k,t)}+i\frac{t}{\gcd(k,t)}z \right)},
\end{align*} 
where we note that $\gcd(h\frac{t}{\gcd(k,t)}, \frac{k}{\gcd(k,t)}) = 1$. Thus we may use \eqref{Eqn: transform usual pochham} with $k \mapsto \frac{k}{\gcd(k,t)}, h \mapsto h \frac{t}{\gcd(k,t)}, z \mapsto \frac{t}{\gcd(k,t)}z$ to obtain
\begin{multline}\label{Eqn: transformation of q^t}
\left(q^t;q^t\right)_\infty = \omega_{h \frac{t}{\gcd(k,t)},\frac{k}{\gcd(k,t)}}^{-1} \left(\frac{t}{\gcd(k,t)}z\right)^{-\frac{1}{2}} e^{\frac{\pi \gcd(k,t)}{12k} \left( \frac{t}{\gcd(k,t)}z -\frac{\gcd(k,t)}{t z} \right)} \\
\times \left(e^{\frac{2\pi i \gcd(k,t)}{k} \left(h_{k,t} +  i\frac{\gcd(k,t)}{tz}\right) } ; e^{\frac{2\pi i \gcd(k,t)}{k} \left(h_{k,t} +  i\frac{\gcd(k,t)}{tz}\right) }\right)_\infty,
\end{multline}
where $0 \leq h_{k,t} < \frac{k}{\gcd(k,t)}$ is defined by $ h \frac{t}{\gcd(k,t)} h_{k,t} \equiv -1 \pmod{\frac{k}{\gcd(k,t)}}$.

Similarly, for $\left(\zeta_b^rq^t;\zeta_b^rq^t\right)_\infty$ the proof of Theorem \ref{Theorem1} (2) implies that we may use \eqref{Eqn: transform usual pochham} with  $h \mapsto \frac{hbt+rk}{\lambda_{t,r,b,h,k}}, k \mapsto \frac{kb}{\lambda_{t,r,b,h,k}}, z \mapsto \frac{tbz}{\lambda_{t,r,b,h,k}}$ and obtain
\begin{multline}\label{Eqn: transformation of zeta q^t}
\left(\zeta_b^rq^t;\zeta_b^rq^t\right)_\infty = \omega_{\frac{hbt+rk}{\lambda_{t,r,b,h,k}}, \frac{kb}{\lambda_{t,r,b,h,k}}}^{-1} \left(\frac{tbz}{\lambda_{t,r,b,h,k}}\right)^{-\frac{1}{2}} e^{\frac{\pi \lambda_{t,r,b,h,k}}{12kb} \left( \frac{tbz}{\lambda_{t,r,b,h,k}} - \frac{\lambda_{t,r,b,h,k}}{tbz} \right) } \\
\times \left( e^{\frac{2\pi i \lambda_{t,r,b,h,k}}{kb} \left( h_{k,t,b,r} + i\frac{\lambda_{t,r,b,h,k}}{tbz} \right)} ; e^{\frac{2\pi i \lambda_{t,r,b,h,k}}{kb} \left( h_{k,t,b,r} + i\frac{\lambda_{t,r,b,h,k}}{tbz} \right)}  \right)_\infty,
\end{multline}
where $0 \leq h_{k,t,b,r} < \frac{kb}{\lambda_{t,r,b,h,k}}$ is defined by $ \frac{hbt+rk}{\lambda_{t,r,b,h,k}} h_{k,t,b,r} \equiv -1 \pmod{\frac{kb}{\lambda_{t,r,b,h,k}}}$.

Combining \eqref{Eqn: transform usual pochham}, \eqref{Eqn: transformation of q^t}, and \eqref{Eqn: transformation of zeta q^t} yields
\begin{multline}\label{eqn: transform}
H_t\left(\zeta_b^r;q\right)
= \Omega_{b,t}(r;h,k) \left( \frac{\gcd(k,t)b}{\lambda_{t,r,b,h,k}}\right)^{\frac{t}{2}} z^{\frac{1}{2}} e^{\frac{\pi}{12k} \left(-z + \left(1-\gcd(k,t)^2 + \frac{\lambda_{t,r,b,h,k}^2}{b^2} \right) \frac{1}{z}\right) } \\
  \times \frac{  \left(e^{\frac{2\pi i \gcd(k,t)}{k} \left(h_{k,t} +  i\frac{\gcd(k,t)}{tz}\right) } ; e^{\frac{2\pi i \gcd(k,t)}{k} \left(h_{k,t} +  i\frac{\gcd(k,t)}{tz}\right) }\right)_\infty^t  }{ \left( e^{\frac{2\pi i \lambda_{t,r,b,h,k}}{kb} \left( h_{k,t,b,r} + i\frac{\lambda_{t,r,b,h,k}}{tbz} \right)} ; e^{\frac{2\pi i \lambda_{t,r,b,h,k}}{kb} \left( h_{k,t,b,r} + i\frac{\lambda_{t,r,b,h,k}}{tbz} \right)}  \right)_\infty^t 
  	\lp e^{\frac{2\pi i}{k}\lp h'+\frac iz \rp}; e^{\frac{2\pi i}{k}\lp h'+\frac{i}{z} \rp} \rp_\infty},
\end{multline}
where
\begin{align*}
\Omega_{b,t}(r;h,k) \coloneqq \frac{\omega_{\frac{hbt+rk}{\lambda_{t,r,b,h,k}}, \frac{kb}{\lambda_{t,r,b,h,k}}}^{t} \omega_{h,k}}{\omega_{h \frac{t}{\gcd(k,t)},\frac{k}{\gcd(k,t)}}^{t}}.
\end{align*}
As usual, we define $P_t(q):=(q;q)_\infty^t=:\sum_{n=0}^\infty q_t(n)q^n$, and $P(q)^t=:\sum_{n=0}^\infty p_t(n)q^n$. Then we see that the principal part of \eqref{eqn: transform} is governed by the sum
\begin{equation*}\label{eqn: principal part}
  \sum_{\substack{n_1,n_2,n_3 \geq 0 \\ r_{k,h,t,b}(n_1,n_2,n_3)\geq 0}} q_t(n_1) p_t(n_2)  p(n_3) \zeta_{kb}^{\gcd(k,t)b h_{k,t}  n_1 + \lambda_{t,r,b,h,k} h_{k,t,b,r} n_2 + bh'n_3}  e^{\frac{\pi}{12 kz} r_{k,h,t,b}(n_1,n_2,n_3) },
\end{equation*}
where
\begin{align*}
r_{k,h,t,b}(n_1,n_2,n_3) \coloneqq 1-\gcd(k,t)^2 + \frac{\lambda_{t,r,b,h,k}^2}{b^2} - 24 \left( \frac{\gcd(k,t)^2}{t} n_1 + \frac{\lambda_{t,r,b,h,k}^2}{tb^2}n_2 + n_3 \right).
\end{align*}
 We denote the Fourier coefficients of $H_t(\zeta_b^r;q)$ by $c_{t,b,r}(n)$. Using Theorem \ref{Thm: Zuckerman} we conclude that 
\begin{multline}\label{Eqn: exact formula}
c_{t,b,r}(n) = \frac{2\pi}{n^{\frac{3}{4}}} b^{\frac{t}{2}} \sum_{k = 1}^\infty \frac{\gcd(k,t)^{\frac{t}{2}}}{k} \sum_{\substack{0 \leq h < k \\ \gcd(h,k)=1}} \Omega_{b,t}(r; h,k) e^{- \frac{2\pi i n h}{k}} \lambda_{t,r,b,h,k}^{-\frac{t}{2}} \sum_{\substack{n_1,n_2,n_3 \geq 0 \\ r_{k,h,t,b}(n_1,n_2,n_3) \geq 0}} q_t(n_1) p_t(n_2)  p(n_3) \\ \times \zeta_{kb}^{\gcd(k,t)b h_{k,t}  n_1 + \lambda_{t,r,b,h,k} h_{k,t,b,r} n_2 + bh'n_3}   \left(\frac{r_{k,h,t,b}(n_1,n_2,n_3)}{24}\right)^{\frac34} I_{\frac{3}{2}} \left(\frac{\pi}{k}\sqrt{\frac{2nr_{k,h,t,b}(n_1,n_2,n_3)}{3}} \right).
\end{multline}

 Since $x^\alpha I_{\alpha}(x)$ is monotonically increasing as $x \rightarrow \infty$ for any fixed $\alpha$, the terms which dominate asymptotically are those which have the largest possible value of $\frac1k \sqrt{r_{k,h,t,b}(n_1, n_2, n_3)}$. In particular for this we require $n_1 = n_2 = n_3 = 0$.
Note that we have $q_t(0) = p_t(0) = p(0) = 1$.
 Since the expression in question is positive we can maximize its square, that is we maximize
\begin{equation*}\label{max}
\dfrac{r_{k,h,t,b}(0,0,0)}{k^2} = \dfrac{1}{k^2} \left( 1 - \gcd(k,t)^2 + \dfrac{\lambda_{t,r,b,h,k}^2}{b^2} \right).
\end{equation*}

We consider the three possible values of $\lambda_{t,r,b,h,k}$. If $\lambda_{t,r,b,h,k} = \gcd(k,t)$, then
$$\dfrac{r_{k,h,t,b}(0,0,0)}{k^2} = \dfrac{1}{k^2} \left( 1 + \left( \dfrac{1}{b^2} - 1 \right) \gcd(k,t)^2 \right) \leq  \left(1+ \left(\frac{1}{9}-1\right)\right) < 1.$$
If $\lambda_{t,r,b,h,k} = b \gcd(k,t)$, then (noting that in this case $k>1$)
$$\dfrac{r_{k,h,t,b}(0,0,0)}{k^2} = \dfrac{1}{k^2} < 1.$$
 Finally, if $\lambda_{t,r,b,h,k} = b^2 \gcd(k,t)$, then we have
$$\dfrac{r_{k,h,t,b}(0,0,0)}{k^2} = \frac{1}{k^2}\left(1 + \left(b^2 - 1\right) \gcd(k,t)^2\right).$$
Since $b \mid\mid \dfrac{k}{\gcd(k,t)}$ in this case, we may write $\gcd(k,t) = b^\varrho d$ where $\gcd(b,d) = 1$, $b^\varrho \mid \mid t$, and $k = b^{\varrho + 1} d k_0$ for $\gcd(k_0, \frac{t}{\gcd(k,t)}) = \gcd(k_0,b) = 1$. Therefore,
$$\dfrac{r_{k,h,t,b}(0,0,0)}{k^2} = \dfrac{1 + \left(b^2-1\right) b^{2\varrho} d^2}{b^{2\varrho + 2} d^2 k_0^2},$$
which is maximized if $k_0 = 1$. In this case, we have $k = b \gcd(k,t)$ and therefore we may write
$$\dfrac{r_{k,h,t,b}(0,0,0)}{k^2} = \dfrac{1 + \left(b^2 - 1\right) \gcd(k,t)^2}{b^2 \gcd(k,t)^2} = \dfrac{b^2-1}{b^2}+\dfrac{1}{b^2 \gcd(k,t)^2}.$$
To maximize this, we need to minimize $\gcd(k,t)$, which is $\gcd(k,t) = 1$. Note that in this case
\begin{align*}
\dfrac{r_{k,h,t,b}(0,0,0)}{k^2} = 1 .
\end{align*}

Since $ht+r \equiv 0 \pmod b$, we have
\begin{align*}
\Omega_{b,t}(r;h,b) = \dfrac{\omega_{\frac{ht+r}{b}, 1}^{t} \omega_{h,b}}{\omega_{ht,b}^{t}} = \dfrac{\omega_{-r\bar{t},b}}{\omega_{-r,b}^{t}},
\end{align*}
where $\bar{t}$ denotes the inverse of $t \pmod b$. Then  by \eqref{Eqn: exact formula} we have
\begin{align*}
c_{t,b,r}(n) &\sim \dfrac{2\pi b^{\frac t2} \omega_{-r\bar{t},b} e^{\frac{2\pi i n r\bar{t}}{b}}}{(24n)^{\frac 34} \omega_{-r,b}^{t} b^{t+1}} I_{\frac 32} \left( \pi \sqrt{\frac{2n}{3}} \right)   \sim \dfrac{e^{\pi\sqrt{\frac{2n}{3}}}}{4\sqrt{3} n b^{\frac{t}{2}+1}} \dfrac{\omega_{-r\bar{t},b}}{ \omega_{-r,b}^{t}} e^{\frac{2\pi i n r\bar{t}}{b}} ,
\end{align*}
as $n \to \infty$, where we use that $I_\alpha(x) \sim \frac{e^x}{\sqrt{2\pi x}}$ as $x \rightarrow \infty$. 
Using \eqref{HR}, we obtain
\begin{align*}
\dfrac{c_{t,b,r}(n)}{p(n)} \sim \begin{cases}
 \dfrac{1}{b^{\frac{t}{2}+1}} \dfrac{\omega_{-r\bar{t},b}}{\omega_{-r,b}^{t}} e^{\frac{2\pi i n r\bar{t}}{b}} &\text{if } b \centernot | t, \\
0 &\text{otherwise.}
\end{cases}
\end{align*}
By \eqref{eqn: split}, we have
\begin{align*}
p_t(a,b;n) = \dfrac{1}{b}p(n) + \dfrac{1}{b} \sum_{r = 1}^{b-1} \zeta_b^{-ar} c_{t,b,r}(n),
\end{align*}
and so dividing through by $p(n)$ yields
\begin{align*}
\dfrac{p_t(a,b;n)}{p(n)} = \dfrac{1}{b} + \dfrac{1}{b} \sum_{r=1}^{b-1} \zeta_b^{-ar} \dfrac{c_{t,b,r}(n)}{p(n)} \sim \begin{cases}
\dfrac{1}{b} + \dfrac{1}{b^{\frac{t}{2} + 2}} \sum\limits_{r=1}^{b-1}  \dfrac{\omega_{-r\bar{t},b}}{\omega_{-r,b}^{t}} \zeta_b^{\lp n \bar{t}-a \rp r} & \text{ if } b \centernot | t, \\[+0.2cm]
\dfrac{1}{b} & \text{ otherwise}
\end{cases}
\end{align*}
as $n \to \infty$. This completes the proof in the case where $b | t$. Otherwise, setting $h = -r\bar{t}$ shows
\begin{align*}
\dfrac{p_t(a,b;n)}{p(n)} \sim \dfrac{1}{b} + \dfrac{1}{b^{\frac{t}{2} + 2}} \sum_{h=1}^{b-1}  \dfrac{\omega_{h,b}}{\omega^t_{th,b}}\zeta_b^{(at-n)h} = \dfrac{1}{b} \left( 1 + \dfrac{K(a,b,t;n)}{b^{\frac{t}{2} + 1}} \right)
\end{align*}
as $n \to \infty$. The evaluation of $K(a,b,t;n)$ in Proposition \ref{Kloosterman Sum Evaluation} then completes the proof.
\end{proof}

\begin{proof}[Proof of Corollary \ref{Corollary3}]To derive Corollary~\ref{Corollary3}, it is enough to consider the leading constants in Theorem~\ref{Theorem2}. Namely, it suffices to show that for $a,b$ fixed, $c_t(a,b;n)$ depends only on $n \pmod{b}$, which is clear from the definition of \eqref{XXX}. 
\end{proof}

\section{Proof of Theorem~\ref{Theorem4} and Corollary~\ref{Corollary5}}\label{Section4}
Here we recall the relevant generating functions for the Poincar\'e polynomials of the Hilbert schemes that pertain to Theorem~\ref{Theorem4}.

\subsection{Work of G\"ottsche and Buryak, Feigin and Nakajima}
For convenience, we let $P(X;T)$ be the usual {\it Poincar\'e polynomial}
\begin{equation*}
P(X;T):=\sum_{j}b_j(X)T^j =\sum_{j} \dim \left( H_j(X,\Q)\right) T^j,
\end{equation*}
which is the generating function for the Betti numbers of $X$.
For the various Hilbert schemes on $n$ points we consider, the work of G\"ottsche, Buryak, Feigin, and
Nakajima \cite{BuryakFeigin, IMRN, Gottsche, GottscheICM} offers the generating function
of these Poincar\'e polynomials as a formal power series in $q$. Namely, we have the following.

\begin{theorem}{\text {\rm (G\"ottsche)}}\label{HilbertGenFcn}
We have that
$$
G(T;q):=\sum_{n=0}^{\infty} P\left(\left(\C^2\right)^{[n]};T \right)q^n=\prod_{m=1}^{\infty}\frac{1}{1-T^{2m-2}q^m}=\frac{1}{ F_3(T^2; q)}.
$$
\end{theorem}

\begin{theorem}{\text {\rm (Buryak and Feigin)}}\label{QuasiHilbertGenFcn}
If $\alpha, \beta\in \N$ are relatively prime, then we have that
$$
G_{\alpha,\beta}(T;q):=\sum_{n=0}^{\infty} P\left(\left(\left(\C^2\right)^{[n]}\right)^{T_{\alpha,\beta}};T\right)q^n=
\frac{1}{F_1(T^2;q^{\alpha+\beta})} \prod_{m=1}^{\infty}\frac{1-q^{(\alpha+\beta)m}}{1-q^m}.
$$
\end{theorem}

\begin{remark}
The Poincar\'e polynomials in these cases only have even degree terms.  The odd index
Betti numbers are always zero. Moreover, letting $T=1$ in these generating functions give Euler's generating function for $p(n).$ Therefore, we directly see that
$$
p(n)=P\left(\left(\C^2\right)^{[n]};1 \right)=P\left(\left(\left(\C^2\right)^{[n]}\right)^{T_{\alpha,\beta}};1\right),
$$
confirming (\ref{decompositionBetti}). Of course, the proofs of these theorems begin with partitions of size $n$.
\end{remark}

Arguing as in the proof of Corollary~\ref{ptabGenFunctions}, we obtain the following generating functions
for the modular sums of Betti numbers.

\begin{corollary}\label{HilbertModularGeneratingFunctions}
For $0\leq a<b$, the following are true.

\noindent
\normalfont{(1)} We have that
$$
\sum_{n=0}^{\infty}B\left(a,b; \left(\C^2\right)^{[n]}\right)q^n=\frac{1}{b}\sum_{r=0}^{b-1}\zeta_b^{-ar}G(\zeta_b^r;q).
$$

\noindent
\normalfont{(2)} If $\alpha, \beta\in \N$ are relatively prime, then we have
$$
\sum_{n=0}^{\infty}B\left(a,b; \left(\left(\C^2\right)^{[n]}\right)^{T_{\alpha,\beta}}\right)q^n=\frac{1}{b}\sum_{r=0}^{b-1}\zeta_b^{-ar}G_{\alpha,\beta}(\zeta_b^r;q).
$$

\end{corollary}

\subsection{Wright's variant of the Circle Method}

The classical Circle Method, as utilized by Hardy--Ramanujan and many others, derives asymptotic or exact formulas for the Fourier coefficients of $q$-series by leveraging modular properties of the generating functions. More recently, a variation of the Circle Method due to Wright has grown increasingly important in number theory. For the proof of Theorem \ref{Theorem4} and Corollary \ref{Corollary5}, we use Wright's variation, which obtains asymptotic formulas for generating functions carrying suitable analytic properties.

\begin{remark}
	 Ngo and Rhoades \cite{NgoRhoades} proved a more restricted version\footnote{We note that hypothesis 4 in Proposition 1.8 of \cite{NgoRhoades} is stated differently than our hypothesis 2 in Proposition \ref{WrightCircleMethod} below.} of the following proposition where the generating function $F$ split as two functions. Our purposes do not require such a splitting, and so we state the proposition in terms of a single function $F$.
\end{remark}

\begin{prop} \label{WrightCircleMethod}
	Suppose that $F(q)$ is analytic for $q = e^{-z}$ where $z=x+iy \in \C$ satisfies $x > 0$ and $|y| < \pi$, and suppose that $F(q)$ has an expansion $F(q) = \sum_{n=0}^\infty c(n) q^n$ near 1. Let $c,N,M>0$ be fixed constants. Consider the following hypotheses:
	
	\begin{enumerate}[leftmargin=*]
		\item[\rm(1)] As $z\to 0$ in the bounded cone $|y|\le Mx$ (major arc), we have
		\begin{align*}
			F(e^{-z}) = z^{B} e^{\frac{A}{z}} \left( \sum_{j=0}^{N-1} \alpha_j z^j + O_\delta\left(|z|^N\right) \right),
		\end{align*}
		where $\alpha_s \in \C$, $A\in \R^+$, and $B \in \R$. 
		
		\item[\rm(2)] As $z\to0$ in the bounded cone $Mx\le|y| < \pi$ (minor arc), we have 
		\begin{align*}
			\lvert	F(e^{-z}) \rvert \ll_\delta e^{\frac{1}{\mathrm{Re}(z)}(A - \kappa)}.
		\end{align*}
		for some $\kappa\in \R^+$.
	\end{enumerate}
	If  {\rm(1)} and {\rm(2)} hold, then as $n \to \infty$ we have for any $N\in \R^+$ 
	\begin{align*}
		c(n) = n^{\frac{1}{4}(- 2B -3)}e^{2\sqrt{An}} \lp \sum\limits_{r=0}^{N-1} p_r n^{-\frac{r}{2}} + O\left(n^{-\frac N2}\right) \rp,
	\end{align*}
	where $p_r := \sum\limits_{j=0}^r \alpha_j c_{j,r-j}$ and $c_{j,r} := \dfrac{(-\frac{1}{4\sqrt{A}})^r \sqrt{A}^{j + B + \frac 12}}{2\sqrt{\pi}} \dfrac{\Gamma(j + B + \frac 32 + r)}{r! \Gamma(j + B + \frac 32 - r)}$. 
\end{prop}

\begin{proof}
	By Cauchy's Theorem, we have 
	$$
		c(n)=\frac{1}{2\pi i} \int_{\mathcal{C}} \frac{F(q)}{q^{n+1}} dq,
	$$
	where $\mathcal{C}$ is a circle centered at the origin inside the unit circle surrounding zero exactly once counterclockwise. We choose $|q|=e^{-\lambda}$  with $\lambda:=\sqrt{\frac{A}{n}}$. Set 
	$$
		A_j(n):= \frac{1}{2\pi i} \int_{\mathcal{C}_1} \frac{z^{B+j} e^{\frac Az}}{q^{n+1}}dq,
	$$
	where $\mathcal{C}_1$ is the major arc. We claim that
	\begin{equation}\label{main a}
		c(n)= \sum_{j=0}^{N-1} \alpha_j A_j(n) + O\lp n^{\frac12 (-B-N-1)} e^{2\sqrt{An}} \rp.
	\end{equation}
	For this write
	$$
		c(n)-\sum_{j=0}^{N-1} \alpha_j A_j(n)=\mathcal{E}_1(n) + \mathcal{E}_2(n),
	$$
	where
	\begin{align*}
		\mathcal{E}_1(n):=\frac{1}{2\pi i} \int_{\mathcal{C}_2} \frac{F(q)}{q^{n+1}} dq, \quad
		\mathcal{E}_2(n):=\frac{1}{2\pi i} \int_{\mathcal{C}_1} \lp F(q)z^{-B} e^{-\frac Az}  - \sum_{j=0}^{N-1} \alpha_jz^j\rp z^{B} e^{\frac Az} q^{-n-1}dq,
	\end{align*}
	where $\mathcal{C}_2$ is the minor arc.
	
	We next bound $\mathcal{E}_1(n)$ and $\mathcal{E}_2(n)$. For $\mathcal{E}_2(n)$ we have, by condition (1)
	$$
		\left| F\lp e^{-z}\rp z^{-B}e^{-\frac Az}-\sum_{j=0}^{N-1}\alpha_j z^j  \right| \ll_\delta |z|^N.
	$$
	Note that on $\mathcal{C}$, $x=\lambda$ and that
	$$
		\left| \exp \lp \frac Az+nz \rp\right| \leq \exp\lp 2 \sqrt{An}\rp.
	$$
	Since the length of $\mathcal{C}_1$ is $\approx \lambda$, we obtain
	$$
		\mathcal{E}_2(n) \ll \lambda |z|^{N+B} \exp\lp 2\sqrt{An} \rp.
	$$
	On $\mathcal{C}_1$, we have $y \ll \lambda$, implying $|z|\sim\frac{1}{\sqrt{n}}$. This gives $\mathcal{E}_1(n)$ satisfies the bound required in \eqref{main a}.
	
	On $\mathcal{C}_2$, we estimate
	$$
		|F(q)| \ll e^{\frac{1}{\lambda}(A - \kappa)}.
	$$ 
	Therefore,
	$$
		\mathcal{E}_1(n) \ll |F(q)| |q|^{-n} \ll e^{\frac{1}{\lambda}(A - \kappa) + n\lambda} \ll e^{(2 - \kappa) \sqrt{An}}.
	$$
	The required bound \eqref{main a} follows. Using Lemma 3.7 of \cite{NgoRhoades} to estimate the integrals $A_j(n)$ now gives the claim.
\end{proof}

\subsection{Proof of Theorem~\ref{Theorem4} and Corollary \ref{Corollary5}} We now apply the Circle Method to the generating functions
in Theorems~\ref{HilbertGenFcn} and \ref{QuasiHilbertGenFcn}.

\begin{proof}[Proof of Theorem \ref{Theorem4}]
	Using first Corollary \ref{HilbertModularGeneratingFunctions} (1) and then Theorem \ref{HilbertGenFcn}, we obtain
	\[
		H_{a,b}(q) := \sum_{n=0}^\infty B\left(a,b;\left(\C^2\right)^{[n]}\right)q^n = \frac1b\left(1+\delta_{2\mid b}\right) \frac{1}{(q;q)_\infty} + \frac1b \sum_{\substack{1\le r\le b-1\\r\ne\frac b2}} \zeta_b^{-ar} \frac{1}{F_3\left(\zeta_b^{2r};q\right)}.
	\]
	We want to apply Proposition \ref{WrightCircleMethod}. For this we first show ($M>0$ arbitrary) that we have as $z\to0$ on the major arc $|y|\le Mx$
	\begin{equation}\label{major}
		H_{a,b}\left(e^{-z}\right) = \frac1b\left(1+\delta_{2\mid b}\right) \sqrt{\frac{z}{2\pi}} e^{\frac{\pi^2}{6z}} (1+O(|z|)).
	\end{equation}
	Recall that we have $P(q):=\sum_{n=0}^\infty p(n)q^n=(q;q)_\infty^{-1}$. First we note the well-known bound (for $|y|\le Mx$, as $z\to0$)
	\[
		P\left(e^{-z}\right) = \sqrt{\frac{z}{2\pi}} e^{\frac{\pi^2}{6z}} (1+O(|z|)).
	\]
	Next we consider $\frac{1}{F_3(\zeta_b^{2r};q)}$ for $\zeta_b^{2r}\ne1$ on the major arc. By Theorem \ref{Theorem1} (3)
	\[
		\frac{1}{F_3\left(\zeta_b^{2r};e^{-z}\right)} = \frac{\left(b^2z\right)^{\frac1b-\frac12}\Gamma\left(\frac1b\right)}{\sqrt{2\pi}} \prod_{j=1}^{b-1} \left(1-\zeta_b^{2rj}\right)^{\frac jb} e^{\frac{\pi^2}{6b^2z}}(1+O(|z|)) \ll |z|^{-N} e^\frac{\pi^2}{6z}
	\]
	for any $N\in\N$. This gives \eqref{major}.
	
	Next we show that we have as $z\to0$ on the minor arc $|y|\ge Mx$
	\begin{equation}\label{minor}
		H_{a,b}\left(e^{-z}\right) \ll e^{\left(\frac{\pi^2}{6}-\kappa\right)\frac1x}.
	\end{equation}
	It is well-known (and follows by logarithmic differentiation) that for some $\mathcal{C}>0$
	\[
		\left|P\left(e^{-z}\right)\right| \le x^{\frac12} e^{\frac{\pi}{6x}-\frac{\mathcal{C}}{x}}.
	\]
	We are left to bound $\frac{1}{F_3(\zeta_b^{2r};q)}$ on the minor arc. For this we write
	\[
		\Log\left(\frac{1}{F_3\left(\zeta_b^{2r};q\right)}\right) = \sum_{m=1}^\infty \frac{q^m}{m\left(1-\zeta_b^{2rm}q^m\right)}.
	\]
	Noting that $|1-\zeta_b^{2rm}q^m|\ge1-|q|^m$, we obtain
	\[
		\left|\mathrm{Log}\left(\frac{1}{F_3\left(\zeta_b^{2r};q\right)}\right)\right| \le \left|\frac{q}{1-\zeta_b^{2r}q}\right|-\frac{|q|}{1-|q|}+\log(P|q|)
	\]
	so we are done once we show that
	\[
		\left|\frac{q}{1-\zeta_b^{2r}q}\right| - \frac{|q|}{1-|q|} < -\frac{\mathcal{C}}{x}
	\]
	for some $\mathcal{C}>0$. Note that
	\[
		\frac{1}{1-\zeta_b^{2r}q} = O_{b,r}(1)
	\]
	and thus
	\[
		\left|\frac{q}{1-\zeta_b^{2r}q}\right| - \frac{|q|}{1-|q|} = -\frac1x+O_{b,r}(1)
	\]
	giving \eqref{minor}. The claim of (1) now follows by Proposition \ref{WrightCircleMethod}.\\
	(2) By Corollary \ref{HilbertModularGeneratingFunctions} (2) and Theorem \ref{QuasiHilbertGenFcn} we have
	\[
		\mathcal{H}_{a,b,\alpha,\beta}(q) := \sum_{n=0}^\infty B\left(a,b;\left(\left(\C^2\right)^{[n]}\right)^{T_{\alpha,\beta}}\right) q^n = \frac1b (1+\delta_{2\mid b}) P(q) + \frac1b \sum_{\substack{1\le r\le b-1\\r\ne\frac b2}} \zeta_b^{-ar} \frac{\left(q^{\alpha+\beta};q^{\alpha+\beta}\right)_\infty}{F_1\left(\zeta_b^{2r};q^{\alpha+\beta}\right)(q;q)_\infty}.
	\]
	We show the same bounds as in (1) with the only additional condition that
	\begin{align}\label{eqn: definition of M}
		M <  \frac{2\pi^2}{b^2} \min_{1\le r<\frac b2} \frac{r(b-2r)}{\left|\sum_{n=1}^\infty \frac{\sin\left(\frac{4\pi r}{b}\right)}{n^2}\right|}.
	\end{align}
	We only need to prove the bounds for
	\[
		\mathcal{H}_{\alpha,\beta}(q) := \frac{\left(q^{\alpha+\beta};q^{\alpha+\beta}\right)_\infty}{F_1\left(\zeta_b^{2r};q^{\alpha+\beta}\right)(q;q)_\infty}.
	\]
	for $\zeta_b^{2r}\ne1$. We may assume without loss of generality that $1\le2r<b$. We start by showing the major arc bound. By Theorem \ref{Theorem1} (1) and \eqref{Eta asymptotic}, we have, for $z$ on the major arc
	\[
		\mathcal{H}_{\alpha,\beta}(q) \ll \left|e^{\frac{\pi^2}{6z}-\frac{\pi^2}{6(\alpha+\beta)z}+\frac{\zeta_b^{2r}\phi\left(\zeta_b^{2r},2,1\right)}{(\alpha+\beta)z}}\right|.
	\]
	So to prove the major arc bound we need to show that for some $\varepsilon>0$
	\[
		\left(\frac{\pi^2}{6}-\varepsilon\right) \mathrm{Re}\left(\frac1z\right) - \mathrm{Re}\left(\frac{\zeta_b^{2r}\phi\left(\zeta_b^{2r},2,1\right)}{z}\right) > 0.
	\]
	We first rewrite
	\[
		\zeta_b^{2r} \phi\left(\zeta_b^{2r},2,1\right) = \sum_{n=1}^\infty \frac{\cos\left(\frac{4\pi rn}{b}\right)+i\sin\left(\frac{4\pi rn}{b}\right)}{n^2}.
	\]
	Now note the evaluation for $0\le\theta\le2\pi$ (see e.g. \cite{Za})
	\[
		\sum_{n=1}^\infty \frac{\cos(n\theta)}{n^2} = \frac{\pi^2}{6}-\frac{\theta(2\pi-\theta)}{4}.
	\]
	Thus we are left to show
	\[
		\frac{2\pi^2r}{b^2}(b-2r)x \ge \left|\sum_{n=1}^\infty \frac{\sin\left(\frac{4\pi rn}{b}\right)}{n^2}\right|y.
	\]
	This follows by the definition of $M$ given in \eqref{eqn: definition of M}.
\end{proof}

\begin{proof}[Proof of Corollary~\ref{Corollary5}] This follows immediately from Theorem \ref{Theorem4} and the definition of $d(a,b)$ in \eqref{XXXX}.
\end{proof}

\section{Examples}\label{Section5}
This section includes examples of the main results in this paper.

\subsection{Examples of Theorem~\ref{Theorem2} and Corollary~\ref{Corollary3}} This subsection pertains to Han's $t$-hook generating functions. For convenience, we define the proportion functions
\begin{equation*}
\Psi_t(a,b;n):=\frac{p_t(a,b;n)}{p(n)}.
\end{equation*}
\begin{example} In the case of $t=3$, we find that
\begin{align*}
H_3(\xi;q)=1+q+2q^2+3\xi q^3+(2+3\xi)q^4+(1+6\xi)q^5+\left(2+9\xi^2\right)q^6+\left(6\xi+9\xi^2\right)q^7\\
+\left(1+3\xi+18\xi^2\right)q^8+\dots.
\end{align*}
and the three generating functions $H_3(a,3;q)$ begin with the terms
\begin{displaymath}
\begin{split}
H_3(0,3;q)&=1+q+2q^2+2q^4+q^5+2q^6+q^8+\dots,\\
H_3(1,3;q)&=3q^3+3q^4+6q^5+6q^7+3q^8+\dots,\\
H_3(2,3;q)&=9q^6+9q^7+18q^8+\dots.
\end{split}
\end{displaymath}
Theorem~\ref{Theorem2} implies (independently of $a$) that
$$
p_3(a,3;n)\sim \frac{1}{12\sqrt{3}n} \cdot e^{\pi \sqrt{\frac{2n}3}} \sim \frac{1}{3}\cdot p(n).
$$
The next table illustrates the conclusion of Corollary~\ref{Corollary3}, that the proportions $\Psi_3(a,b;n)
\to \frac 13.$
\medskip

\begin{center}

\begin{tabular}{|c|cc|cc|cc|}
\hline \rule[-3mm]{0mm}{8mm}
$n$       && $\Psi_3(0,3;n)$           && $\Psi_3(1,3;n)$  & $\Psi_3(2,3;n)$ & \\   \hline 
$ 100$ &&  $\approx 0.4356$ && $\approx 0.1639$ & $\approx 0.4003$ & \\
$\ \ \vdots \ \ $ && \vdots &&$\vdots$ &  $\vdots$ & \\
$ 500$ &&  $\approx 0.3234$ && $\approx 0.3670$ & $\approx 0.3096$ & \\
$ 600$ &&  $\approx 0.3318$ && $\approx 0.3114$ & $\approx 0.3567$ & \\
$\ \ \vdots \ \ $ && \vdots &&$\vdots$ &  $\vdots$ & \\
$ 2100$ &&  $\approx 0.3320$ && $\approx 0.3348$ & $\approx 0.3332$ & \\
$ 2300$ &&  $\approx 0.3330$ && $\approx 0.3345$ & $\approx 0.3325$ & \\
$ 2500$ &&  $\approx 0.3324$ && $\approx 0.3337$ & $\approx 0.3339$ & \\
\hline
\end{tabular}
\end{center}
\end{example}
\medskip

\begin{example} We consider a typical case where the modular sums of $t$-hook functions are not equidistributed. 
We consider $t=2$, where we have
\begin{align*}
H_2(\xi;q)=1+q+2\xi q^2+(1+2\xi)q^3+5\xi^2 q^4 +\left(2\xi+5\xi^2\right)q^5+\left(1+10\xi^3\right)q^6+\left(5\xi^2+10\xi^3\right)q^7\\
+\left(2\xi +20\xi^4\right)q^8+\dots.
\end{align*}
The three generating functions $H_2(a,3;q)$ begin with the terms
\begin{displaymath}
\begin{split}
H_2(0,3;q)&=1+q+q^3+11q^6+10q^7+\dots,\\
H_2(1,3;q)&=2q^2+2q^3+2q^5+22q^8+\dots,\\
H_2(2,3;q)&=5q^4+5q^5+5q^7+\dots.
\end{split}
\end{displaymath}
Theorem~\ref{Theorem2} implies that
$$
p_2(a,3;n)\sim \frac{A(a,n)}{12\sqrt{3}n}\cdot  e^{\pi \sqrt{\frac{2n}{3}}} \sim \frac{A(a,n)}{3} \cdot p(n),
$$
where $A(a,n)\in \{0, 1, 2\}$ satisfies the congruence $A(a,n)\equiv 2-a-n\pmod 3.$
This explains the uneven distribution established by
Corollary~\ref{Corollary3} in this case. 
In particular, we have that
$$
\lim_{n\rightarrow \infty}\frac{p_t(a,3; 3n+2-a)}{p(n)}=0.
$$
 Of course, this zero distribution is weaker than the vanishing obtained in Theorem~\ref{Vanishing}.
 
\noindent
The next table illustrates the uneven asymptotics for $n\equiv 0\pmod 3.$
\medskip

\begin{center}

\begin{tabular}{|c|cc|cc|cc|}
\hline \rule[-3mm]{0mm}{8mm}
$n$       && $\Psi_2(0,3;n)$           && $\Psi_2(1,3;n)$  & $\Psi_2(2,3;n)$ & \\   \hline 
$ 300$ &&  $\approx 0.7347$ && $\approx 0.2653$ & $0$ & \\
$\ \ \vdots \ \ $ && \vdots &&$\vdots$ &  $\vdots$ & \\
$ 600$ &&  $\approx 0.6977$ && $\approx 0.3022$ & $0$ & \\
$ 900$ &&  $\approx 0.6837$ && $\approx 0.3163$ & $0$ & \\
$\ \ \vdots \ \ $ && \vdots &&$\vdots$ &  $\vdots$ & \\
$ 4500$ &&  $\approx 0.6669$ && $\approx 0.3330$ & $0$ & \\
$ 4800$ &&  $\approx 0.6669$ && $\approx 0.3330$ & $0$ & \\
$ 5100$ &&  $\approx 0.6668$ && $\approx 0.3331$ & $0$ & \\
\hline
\end{tabular}
\end{center}
\end{example}
\medskip

\begin{example} We consider another typical case where the modular sums of $t$-hook functions are not equidistributed. 
We consider $t=4$, where we have
$$
H_4(\xi;q)=1+q+2q^2+3q^3+(1+4\xi)q^4+(3+4\xi)q^5+(3+8\xi)q^6+(3+12\xi)q^7+\left(4+4\xi+14\xi^2\right)q^8+\dots.
$$
The three generating functions $H_4(a,3;q)$ begin with the terms
\begin{displaymath}
\begin{split}
H_4(0,3;q)&=1+q+2q^2+3q^3+q^4+3q^5+3q^6+3q^7+4q^8+\dots,\\
H_4(1,3;q)&=4q^4+4q^5+8q^6+12q^7+4q^8+\dots,\\
H_4(2,3;q)&=14q^8+\dots.
\end{split}
\end{displaymath}
Theorem~\ref{Theorem2}, restricted to partitions of integers which are multiples of 12, gives 
$$
p_4(a,3;12n)\sim \begin{cases} \frac{4}{9}\cdot p(12n)\ \ \ \ \ &{\text {\rm if $a=0$,}}\\
\frac{1}{3}\cdot p(12n) \ \ \ \ \ &{\text {\rm if $a=1,$}}\\
\frac{2}{9}\cdot p(12n) \ \ \ \ \ &{\text {\rm if $a=2$.}}
\end{cases}
$$
\noindent
The next table illustrates these asymptotics.
\medskip

\begin{center}

\begin{tabular}{|c|cc|cc|cc|}
\hline \rule[-3mm]{0mm}{8mm}
$n$       && $\Psi_4(0,3;12n)$           && $\Psi_4(1,3;12n)$  & $\Psi_4(2,3;12n)$ & \\   \hline 
$ 10$ &&  $\approx 0.4804$ && $\approx 0.3373$ & $\approx 0.1823$ & \\
$\ \ \vdots \ \ $ && \vdots &&$\vdots$ &  $\vdots$ & \\
$ 50$ &&  $\approx 0.4500$ && $\approx 0.3381$ & $\approx 0.2119$ & \\
$ 60$ &&  $\approx 0.4485$ && $\approx 0.3373$ & $\approx 0.2142$ & \\
$\ \ \vdots \ \ $ && \vdots &&$\vdots$ &  $\vdots$ & \\
$ 180$ &&  $\approx 0.4447$ && $\approx 0.3340$ & $\approx 0.2212$ & \\
$ 190$ &&  $\approx 0.4447$ && $\approx 0.3339$ & $\approx 0.2214$ & \\
$ 200$ &&  $\approx 0.4446$ && $\approx 0.3338$ & $\approx 0.2215$ & \\
\hline
\end{tabular}
\end{center}

\end{example}

\subsection{Examples of Theorem~\ref{Theorem4} and Corollary~\ref{Corollary5}}
Finally, we consider examples of the asymptotics and distributions in the setting of Hilbert schemes on $n$ points.

\smallskip
\begin{example} By G\"ottsche's Theorem (i.e., Theorem~\ref{HilbertGenFcn}), we have
\begin{displaymath}
\begin{split}
G(T;q)&:=\sum_{n=0}^{\infty} P\left(\left(\C^2\right)^{[n]};T \right)q^n=\prod_{m=1}^{\infty}\frac{1}{1-T^{2m-2}q^m}=\frac{1}{F_3(T^{-2}; T^2q)}\\
&=1+q+\left(1+T^2\right)q^2+\left(1+T^2+T^4\right)q^3+\left(1+T^2+2T^4+T^6\right)q^4+\dots. \\
\end{split}
\end{displaymath}
Theorem~\ref{Theorem4} (1) implies that
$$
B\left(a,3; \left(\C^2\right)^{[n]}\right) \sim \frac{1}{12\sqrt{3}n}\cdot  e^{\pi \sqrt{\frac{2n}{3}}},
$$
and so Corollary~\ref{Corollary5} implies that $\delta(a,3;n)\to \frac 13$. The next table illustrates
this phenomenon.
\medskip

\begin{center}

\begin{tabular}{|c|cc|cc|cc|}
\hline \rule[-3mm]{0mm}{8mm}
$n$       && $\delta(0,3;n)$           && $\delta(1,3;n)$  & $\delta(2,3;n)$ & \\   \hline 
$ 1$ &&  $1$ && $0$ & $0$ & \\
$ 2$ &&  $0.5000$ && $0$ & $0.500$ &\\
$\ \ \vdots \ \ $ && \vdots &&$\vdots$ &  $\vdots$ & \\
$ 18$ &&  $\approx 0.3377$ && $\approx 0.3325$ & $\approx 0.3299$ & \\
$ 19$ &&  $\approx 0.3367$ && $\approx 0.3306$ & $\approx 0.3327$ & \\
$ 20$ &&  $\approx 0.3333$ && $\approx 0.3317$ & $\approx 0.3349$ & \\
\hline
\end{tabular}
\end{center}
\medskip

\end{example}

\begin{example} By Theorem~\ref{QuasiHilbertGenFcn}, for $\alpha=2$ and $\beta=3$ we have
\begin{displaymath}
\begin{split}
G_{2,3}(T;q)&:=\sum_{n=0}^{\infty} P\left(\left(\left(\C^2\right)^{[n]}\right)^{T_{2,3}};T\right)q^n=
\frac{1}{F_1(T^2;q^{5})} \prod_{m=1}^{\infty}\frac{\left(1-q^{5m}\right)}{1-q^m}\\
&=1+q+2q^2+\dots+\left(6+T^2\right)q^5+\left(10+T^2\right)q^6+\left(13+2T^2\right)q^7+\dots.
\end{split}
\end{displaymath}
Theorem~\ref{Theorem4} (2) implies that
$$
	B\left(a,3;  \left(\left (\C^2\right)^{[n]}\right)^{T_{\alpha,\beta}}\right)\sim \frac{1}{12\sqrt{3}n} \cdot e^{\pi \sqrt{\frac{2n}{3}}},
$$
and so Corollary~\ref{Corollary5} yields that $\delta_{2,3}(a,3;n)\to \frac 13$. The next table illustrates
this phenomenon.
\medskip
\begin{center}

\begin{tabular}{|c|cc|cc|cc|}
\hline \rule[-3mm]{0mm}{8mm}
$n$       && $\delta_{2,3}(0,3;n)$           && $\delta_{2,3}(1,3;n)$  & $\delta_{2,3}(2,3;n)$ & \\   \hline 
$ 1$ &&  $1$ && $0$ & $0$ & \\
$ 2$ &&  $1$ && $0$ & $0$ &\\
$\ \ \vdots \ \ $ && \vdots &&$\vdots$ &  $\vdots$ & \\
$ 100$ &&  $\approx 0.3693$ && $\approx 0.2658$ & $\approx 0.3649$ & \\
$ 200$ &&  $\approx 0.3343$ && $\approx 0.3176$ & $\approx 0.3481$ & \\
$ 300$ &&  $\approx 0.3313$ && $\approx 0.3293$ & $\approx 0.3393$ & \\
$ 400$ &&  $\approx 0.3318$ && $\approx 0.3324$ & $\approx 0.3358$ & \\
$ 500$ &&  $\approx 0.3324$ && $\approx 0.3332$ & $\approx 0.3343$ & \\
\hline
\end{tabular}
\end{center}

\end{example}

\end{document}